\documentclass{amsart}
\usepackage{hyperref,amsmath,amsfonts,amssymb,amsthm,enumerate,mathtools,graphicx,xparse,xcolor,tikz,scalerel,stackengine}
\usepackage[a4paper, left=2cm, right=2cm, top=3cm, bottom=2cm]{geometry}
\usepackage[utf8]{inputenc}
\usepackage{enumitem}
\usepackage{float}
\usepackage{mdframed}

\graphicspath{{figures/}}
\usepackage{pgfplots}
\usepackage{subfig}

\usepackage{accents}
\usepackage{empheq}
\usepackage{subfig}
\usepackage{multimedia}
\usepackage{animate}
\usepackage[makeroom]{cancel}
\definecolor{morange}{rgb}{0.92, 0.51, 0.11}
\definecolor{mLightGreen}{HTML}{14B03D}
\definecolor{mPink}{HTML}{e1388d}
\definecolor{mRed}{HTML}{E83C3C}

\usepackage{scalerel}

\tikzstyle{dot}=[draw,fill=white,circle,inner sep=0pt,minimum size=4pt]

\pgfkeys{
  /pgf/declare function={Si1(\r,\nu,\d,\D)=2*\nu*\d*(\D/\d)^(\r-1);},
  /pgf/declare function={Si2(\r,\nu,\d,\D)=2*\nu*(1+\D^2/\d^2)^(\r/2-1)*\D*2^(1-\r/2);},
  /pgf/declare function={Si(\r,\q,\nuratio,\D)=Si1(\r,1/(1+\nuratio),1,\D)+Si2(\q,\nuratio/(1+\nuratio),1,\D);},
  /pgf/declare function={Sii1(\nu,\D)=2*\nu*\D;},
  /pgf/declare function={Dii(\rp,\nu,\d,\Stwo)=1/(2*\nu)*(1+\Stwo^2/(2*\nu*\d)^2)^(\rp/2-1)*\Stwo*2^(1-\rp/2);},
}

\numberwithin{dummy}{section}
\numberwithin{equation}{section}

\theoremstyle{plain}
\newtheorem{thm}{Theorem}[section]
\newtheorem{lemma}[thm]{Lemma}

\theoremstyle{plain}
\newtheorem{df}[thm]{Definition}

\theoremstyle{remark}

\newcommand{\R}{\mathbb{R}}
\newcommand{\N}{\mathbb{N}}
\newcommand{\D}{\mathbb{D}}

\newcommand*\dd{\mathop{}\!\mathrm{d}}

\newcommand*{\ddt}{\mathop{}\!\frac{\mathrm{d}}{\mathrm{d}t}}
\def\di{\mathop{\mathrm{div}}\nolimits}

\NewDocumentCommand{\oldnorm}{sO{}m}{%
  {\IfBooleanTF{#1}
    {\oldnormaux{\left|}{\right|}{#3}}
    {\oldnormaux{#2|}{#2|}{#3}}}
}
\makeatletter
\newcommand{\oldnormaux}[3]{\mathpalette\oldnormaux@i{{#1}{#2}{#3}}}
\newcommand{\oldnormaux@i}[2]{\oldnormaux@ii#1#2}
\newcommand{\oldnormaux@ii}[4]{%
  \sbox\z@{$\m@th#1#2#4#3$}%
  \sbox\tw@{$\m@th\|$}%
  \mathopen{\hbox to\wd\tw@{\hss\vrule height \ht\z@ depth \dp\z@ width .3\wd\tw@\hss}}%
  #4
  \mathclose{\hbox to\wd\tw@{\hss\vrule height \ht\z@ depth \dp\z@ width .3\wd\tw@\hss}}%
}
\makeatother

\newcommand{\vect}[1]{\boldsymbol{#1}}

\newcommand{\Abs}[1]{\left\lVert{#1}\right\rVert}

\def\bv{\vect{v}}

\tikzstyle{dot}=[draw,fill=white,circle,inner sep=0pt,minimum size=4pt]
\tikzstyle{dot}=[draw,fill=white,circle,inner sep=0pt,minimum size=4pt]

\begin{document}

\title{On the well-possedness of time-dependent three-dimensional Euler fluid flows}\thanks{This work is supported by the project 20-11027X financed by Czech science foundation (GA\v{C}R). The authors are  members of the Ne\v{c}as Center for Mathematical Modeling.}
\author[M.~Bul\'{i}\v{c}ek]{Miroslav Bul\'i\v{c}ek}
\address{Charles University, Faculty of Mathematics and Physics, Mathematical Institute\\
Sokolovsk\'{a} 83, 186 75 Prague, Czech Republic}
\email{mbul8060@karlin.mff.cuni.cz}

\author[J. M\'{a}lek]{Josef M\'{a}lek}
\address{Charles University, Faculty of Mathematics and Physics, Mathematical Institute \\
Sokolovsk\'{a} 83, 186 75 Prague, Czech Republic}
\email{malek@karlin.mff.cuni.cz}

\begin{abstract}
    We study the mathematical properties of time-dependent flows of incompressible fluids that respond as an Euler fluid until the modulus of the symmetric part of the velocity gradient exceeds a certain, a-priori given but arbitrarily large, critical value. Once the velocity gradient exceeds this threshold, a dissipation mechanism is activated. Assuming that the fluid, after such an activation, dissipates the energy in a specific manner, we prove that the corresponding initial-boundary-value problem is  globally-in-time well-posed in the sense of Hadamard. In particular, we show that for an arbitrary, sufficiently regular, initial velocity there is a unique weak solution to the spatially periodic problem.
\end{abstract}
\keywords{Euler fluid, activated Euler fluid, three-dimensional flow, weak solution, existence, uniqueness, regularity method}
\subjclass[2000]{35K59, 35K92,35D30, 76D03}

\maketitle




\section{Introduction}\label{sec:1}

\label{1.1} \subsection{Statement of the problem and main result}
A (nonlinear) \emph{dynamical system} of a schematic form
$$
\frac{\dd{u}}{\dd{t}} = F(u) \textrm{ in } (0,\infty), \qquad u (0) = u_0,
$$
is \emph{well-posed} in the sense of Hadamard if it admits {a} solution for all $t>0$, this solution is uniquely determined by its initial state $u_0$, and the solution is stable with respect to small perturbations of $u_0$. Towards the goal to establish well-possedness of any dynamical system, the  identification of \emph{the correct concept of solution} is of a fundamental importance.

In (incompressible) fluid mechanics, \emph{two-dimensional} internal flows governed by the \emph{Navier--Stokes equations} serve as an important dynamical system that is well-posed. More precisely, it is known that the initial-boundary-value problem governed by two-dimensional Navier--Stokes equations, when subject to standard boundary conditions such as no slip, Navier's slip, or complete slip, admit \emph{a unique weak solution},  see \cite{Leray34,Hopf51,Lad59,LionsProdi59}. 

In this context,  the \emph{concept of weak solution} is very natural for at least two reasons. First, the governing equations that stem from the application of  Newton's second law of classical mechanics to an arbitrary (open) subset of the domain occupied by the fluid lead to integral identities  that  are equivalent to the weak formulation of the governing equations. This weak formulation is thus a  primary  object while the classical pointwise formulation of the governing equations obtained under additional smoothness of the  quantities involved is, from this perspective, a secondary (derived) concept.\footnote{This resembles the analogous situation concerning the Euler--Lagrange equations in  the calculus of variations. A weak form of these equations is equivalent to the condition stating that the Gateaux derivative of the functional  under consideration vanishes  in each direction; as this necessary condition for the minimizer appears first, the weak formulation of the Euler--Lagrange equations is of a more fundamental nature than  their classical,  pointwise form.}  Secondly, in general, the only available information controlled by the initial total kinetic energy of the system leads to control of the total kinetic energy at any instant of time and of the overall energy dissipated from the system. This piece of information should be incorporated into the concept of a solution that one wishes to deal with. Note that for  the two-dimensional Navier--Stokes equations the information coming from this a~priori energy balance leads to the function spaces in which the weak solution is constructed.

However, for  the \emph{three-dimensional Navier--Stokes equations}, the question of global-in-time
well-posedness is open. More precisely,  a (suitable) weak solution is known to exist for any $t>0$,  see \cite{Leray34, Hopf51,CaKoNi82}, but its uniqueness is an open question.\footnote{In fact, recent results \cite{BCV22,ABC22} regarding the existence of infinitely many weak solutions to the initial-boundary-value problem governed by the Navier--Stokes equations indicate that such a well-posed theory for  the three-dimensional Navier--Stokes equations might not even be  attainable. See also \cite{OAL_nonunique}.} There are, however, three-dimensional dynamical systems describing incompressible fluid flows that are well-posed for initial data with finite total kinetic energy. The purpose of this study is to significantly \emph{enlarge the class of incompressible fluid flow problems that are well-posed}.

To be more specific, we recall that unsteady flows of incompressible fluids are governed by the following equations:\footnote{We omit the effect of external body forces for simplicity here. Also, the second equation (balance of linear momentum) is divided by the constant density; $p$ is then the kinematic pressure. Finally, we note that most researchers  in the field of fluid mechanics prefer to  write $(\bv \cdot \nabla)\bv$ instead of $[\nabla\bv]\bv$.}
\begin{equation} \label{zv_07}
    \di \bv = 0, \quad \partial_t{\bv} + [\nabla \bv]\bv - \di \mathbb{S} = - \nabla p,
\end{equation}
where $\bv=(v_1, v_2, v_3)$ is the velocity, $p$ is, in the context considered here, the mean normal stress (the pressure), and $\mathbb{S}$ is a part of the Cauchy stress $\mathbb{T}$ that enters  into an additional equation (the constitutive equation) that characterizes the inner frictional mechanisms of a particular class of fluids one  is dealing with. Note that  because of the incompressibility constraint $\di\bv=0$  one has that $\mathbb{T} = -p\mathbb{I}+ \mathbb{S}$.

Two of the most popular classes of incompressible fluids are \emph{Navier--Stokes fluids}, characterized by  a linear constitutive equation  relating $\mathbb{S}$ and the symmetric part of the velocity gradient $\mathbb{D}\bv:=\frac{1}{2}(\nabla \bv + (\nabla \bv)^{\mathrm{T}})$, i.e.,
\begin{align}
    \mathbb{S} := 2\nu \mathbb{D}\bv, \label{zv_11}
\end{align}
and \emph{Euler fluids}, satisfying
\begin{align}
    \mathbb{S} := \mathbb{O}. \label{zv_12}
\end{align}
In \eqref{zv_11}, $\nu$ is a positive constant representing the (kinematic) viscosity, while in \eqref{zv_12} $\mathbb{O}$ stands for the zero tensor (i.e., all  of its components vanish). See Fig.~\ref{zv_13} for the graphs of these responses.
\begin{figure}[]
    \centering
    \begin{tabular}{cc}
          \begin{tikzpicture}[scale=2.2]%
        \def\figm{1.1}
        \def\fignua{0.4}
        \def\fignub{0.2}
        \def\figr{6.5}
        \def\figmax{1.45}
        \def\figxmax{1.6}
        \footnotesize
        \draw[ultra thin] (-0.1,  0.0) -- (\figxmax, 0.0) node[below]{$|\mathbb{D}\bv|$};
        \draw[ultra thin] ( 0.0, -0.1) -- (0.0, 1.1) node[left ]{$|\mathbb{S}|$};
        \draw[variable=\t,smooth,ultra thick]
          plot[domain=0:\figxmax] ({\t},{\fignua*\t});
        \end{tikzpicture}
\qquad &      \qquad \begin{tikzpicture}[scale=2.2]%
        \def\figm{1.1}
        \def\figM{1.4}
        \def\figa{1.5}
        \def\figmax{1.34}
        \def\figxmax{1.8}
        \footnotesize
        \draw[ultra thin] (-0.1,  0.0) -- (\figxmax, 0.0) node[below]{$|\mathbb{D}\bv|$};
        \draw[ultra thin] ( 0.0, -0.1) -- (0.0, 1.1) node[left ]{$|\mathbb{S}|$};
        \draw[variable=\t,smooth,ultra thick]
          plot[domain=0:\figxmax] ({\t},{0});
        \end{tikzpicture}
    \end{tabular}
    \caption{Sketches of the responses, i.e., the constitutive equations, that characterize the Navier--Stokes fluids (left) and the Euler fluids (right).}
    \label{zv_13}
\end{figure}
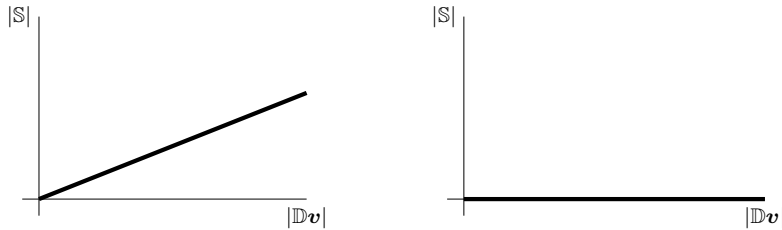

Let us also recall that for suitable boundary conditions (including no-slip, perfect slip, or the  spatially periodic problem), the following formal energy equality holds:
\begin{equation}
    \label{zv_15}
    \frac12 \int_{\Omega} |\bv(t, \cdot)|^2 + \int_0^t \int_{\Omega} \mathbb{S}\cdot\mathbb{D}\bv = \frac12\int_{\Omega} |\bv(0,\cdot)|^2,
\end{equation}
where $\Omega \subset \mathbb{R}^3$ is an open bounded set in which the internal flows (i.e., no inflows/outflows) take place.

We can now return to the question of identification of well-posed three-dimensional problems in incompressible fluid mechanics. This program has been initiated by Lady\v{z}enskaya, see \cite{OAL67}, where she proved that if extra dissipation  of a particular form is added to the Navier--Stokes model, then the problem is well-posed. More precisely,  by
considering
\begin{equation}
    \label{zv_16}
\mathbb{S}:=2\nu \mathbb{D}\bv+ 2\tilde\nu|\mathbb{D}\bv|^{r-2}\mathbb{D}\bv,
\end{equation}
she proved, see \cite{La72}, that the corresponding problem is well-posed if $r>5/2$. Note that \eqref{zv_16} reduces to the Navier--Stokes model if $\tilde\nu=0$ or $r=2$. This result has been generalized and extended in \cite{mnr93,mnrr96, BuEtKaPr10, BuKaPr19}  down to $r\ge 11/5$ under a different condition on the initial velocity.\footnote{This extension result holds for an initial velocity with finite total kinetic energy in the sense of trajectories; in the classical sense the problem is well-posed for initial states that are sufficiently regular.} Recently, see \cite{BuMaMa23}, this result has been extended, for the same range of $r$, to a constitutive equation of the form (see the left sketch in Fig.~\ref{zv_14}):
\begin{equation}
    \label{zv_16a}
\mathbb{S}:=\begin{cases}
        &2\nu\mathbb{D}\bv \hspace{5cm} \hfil \textrm{ if } 0\le |\mathbb{D}\bv| \le m,  \\
        &2\nu\mathbb{D}\bv + 2\tilde\nu(|\D\bv| - m)|\D\bv|^{r-3}\D\bv \qquad \,\,\,\textrm{ if } m< |\mathbb{D}\bv|,
    \end{cases}
\end{equation}
where $m>0$ is the activation that can be in principle so large that its precise value is irrelevant to  the particular fluid mechanics setting one is interested in. We call fluids  that respond according to \eqref{zv_16a} \emph{activated Navier--Stokes fluids}. Although the analysis of the initial-boundary-value problem associated with \eqref{zv_16a} does not  give rise to additional mathematical difficulties compared to the earlier studies cited above, from a modeling and computational point of view, the consequences are significant. For suitable numerical approximations of the initial-boundary-value problem associated with \eqref{zv_16a}, we know that the limiting object that is being computationally approximated is  the unique weak solution of a well-posed problem. Moreover, if for a sequence $(\bv_n)_{n=1}^\infty$ of numerical approximations, converging to the unique weak solution $\bv$ of the initial-boundary-value problem associated with \eqref{zv_16a}, the values of $|\D \bv_n|$ are less than or equal to the activation parameter $m$ for some $n \geq 1$, then $\bv_n$ can be viewed as an approximation to a weak solution of the three-dimensional Navier--Stokes equations.




The question we aim to address affirmatively in this research is whether the same applies to three-dimensional flows  of \emph{Euler fluids}. We will demonstrate that there is an additional dissipation  mechanism in the model, triggered only at high values of the generalized shear rate $ \mathbb{D}\bv $, which ensures that initial-boundary-value problems  for such \emph{activated Euler fluids} are well-posed.

Three-dimensional Euler fluid flows are a major subject of mathematical research  because of the numerous challenges presented by the governing nonlinear PDE system. In this context, it was shown by Daneri, Runa, and Sz\'{e}kelyhidi \cite{DaRuSz21}, building on the convex integration method \cite{LeSz10}, that there exist ``wild" initial conditions that give rise to infinitely many weak solutions satisfying the energy inequality.  Moreover, the set of ``wild" initial conditions is dense within the space of initial velocity states that possess bounded total kinetic energy.

In this work, we offer an alternative perspective. We study a special class of activated Euler fluids of the form (see the sketch in Fig.~\ref{zv_14},  middle image, for  the corresponding graph)
\begin{equation}
    \label{zv_17}
\mathbb{S}:=\left\{\begin{aligned}
        &\mathbb{O} &&\textrm{ if } 0\le |\mathbb{D}\bv| \le m,  \\
        &\sigma\frac{(|\D\bv| - m)}{(M^a - |\D\bv|^a)^{\frac 1a}}\frac{\D\bv}{|\D\bv|} &&\textrm{ if } m< |\mathbb{D}\bv| < M,
    \end{aligned}
    \right.
\end{equation}
and show that for $\sigma>0$ and $a\in(0,1/2)$ the initial-boundary-value problem governed by such fluids is well-posed in the classical sense, for  a sufficiently regular, but  otherwise arbitrary, initial velocity. In addition, the solution satisfies the associated energy equality.

Before we start to formulate the result in a precise way, four remarks are in order. The first is of a general nature. The balance equations in continuum mechanics are derived under certain assumptions and are not meant to describe the fluid behaviour at (extremely) large values for the velocity, its derivative, or the forces. Above, we have modified the response of the classical fluid models (Euler and Navier--Stokes) only for extremely large values of $|\mathbb{D}\bv|$. When solving any particular problem in fluid mechanics involving Euler fluids, it can happen that (if the activation parameter $m$ is large enough) the activation does not take place at all. Then one would get the same flow (the same solution) when solving \eqref{zv_07} with \eqref{zv_17} as for the Euler fluid governed by \eqref{zv_07} and \eqref{zv_12}.

The second comment concerns the appearance of activated Euler fluids. Surprisingly, to the best of our knowledge, these  activated responses  appeared only as an outcome of a systematic classification of implicitly constituted fluids\footnote{The systematic investigation of the properties of \emph{implicitly constituted fluids and solids} was initiated in the studies by Rajagopal~\cite{Ra03, Ra06}.}, see~\cite{BlMaRa20}. For example, the fluid model that  responds as an Euler fluid prior to activation and  as a Navier--Stokes fluid once the activation  has taken  place is a dual\footnote{ This means that the activated Euler fluid model (which behaves as the Navier--Stokes fluid model after activation) is obtained by interchanging the role of $\mathbb{S}$ and $\mathbb{D}\bv$ in the constitutive equation for the Bingham fluid, see \cite{BlMaRa20, Ga-OrMaRa24} for more details.} model to the popular Bingham fluid model. While the Bingham fluid model is used in  a number of different engineering applications, the activated Euler  model has not been used in applications so far, although Prandtl's boundary layer theory or superfluidity  is based on  combining the Euler and Navier--Stokes fluid models. We refer to  the preliminary numerical studies \cite{Ga-OrMaRa24} in which the performance of the activated Euler fluid model is compared with the full Navier--Stokes model, for flows along  an airfoil. We also wish to mention that  the PDE analysis 
of both steady and unsteady internal flows of several classes of activated Euler fluids (but not the one studied here) with different types of boundary conditions  developed in \cite{BlMaRa20} concerns  the existence of weak  solutions; however, for all  of the dynamical problems studied in \cite{BlMaRa20} the  question of uniqueness, and thereby also the question of well-posedness of the initial-boundary-value problems whose weak solution were shown there to exist, remains open. 

Third, although the value of $m>0$ in the response can be  arbitrary,  it is understood that  in the context herein it is taken to be large enough so that  for the specific physical situation of interest the equations for activated Euler fluids (\eqref{zv_07} with \eqref{zv_17}) have the potential to describe the same flow as the  equations for a standard Euler fluid. As   was stated above, one can have, however, situations where Euler and Navier--Stokes fluids  mix and the transition can be also triggered by the specific value  of $|\mathbb{D}\bv|$, say $\underline{m}$, that is supposed to be physically relevant (even small). Then, in order to   have a well-posed problem, one can consider the following fluid response (see also Fig.~\ref{zv_14},  right-most image, for the corresponding sketch):
\begin{equation}
    \label{zv_1717}
\mathbb{S}:=\left\{\begin{aligned}
        &\mathbb{O} &&\textrm{ if } 0\le |\mathbb{D}\bv| \le \underline{m},  \\
        &2\nu(|\mathbb{D}\bv|-\underline{m})\frac{\mathbb{D}\bv}{|\mathbb{D}\bv|} && \textrm{ if } \underline{m} < |\mathbb{D}\bv| \le m, \\
        &2\nu(|\mathbb{D}\bv|-\underline{m})\frac{\mathbb{D}\bv}{|\mathbb{D}\bv|} + \sigma\frac{(|\D\bv| - {m})}{(M^a - |\D\bv|^a)^{\frac 1a}}\frac{\D\bv}{|\D\bv|} &&\textrm{ if } m< |\mathbb{D}\bv| < M.
    \end{aligned}
    \right.
\end{equation}

Finally, it is worth observing that if $m=0$ in \eqref{zv_17} then the response is invertible and one has
$$
\mathbb{S} = \sigma\frac{\D\bv}{(M^a - |\D\bv|^a)^{\frac 1a}} \qquad \iff \qquad \mathbb{D}\bv = M\frac{\mathbb{S}}{(\sigma^a + |\mathbb{S}|^a)^{\frac 1a}}.
$$
When replacing $\mathbb{D}\bv$ by the linearized strain, one obtains the limiting strain model introduced by Rajagopal in \cite{krr2007, krr2010, krr2011}, see also \cite{bmrs2014}. The  analyses of  the corresponding boundary-value problems involving  those models have developed tools and results that we  can also employ for the class of problems explored in this study. We refer in particular to \cite{BeBuMaSu16} and \cite{BuHrMa23}.

In the next section, we introduce the problem and notation, give a definition of a weak solution, and formulate our main result. 
We also describe the structure of the remaining parts of the manuscript.

\begin{figure}[]
    \centering
    \begin{tabular}{ccc}
             \begin{tikzpicture}[scale=2.2]%
        \def\figm{0.8}
        \def\fignua{0.4}
        \def\fignub{0.2}
        \def\figr{6.5}
        \def\figmax{1.38}
        \def\figxmax{1.8}
        \footnotesize
        \draw[ultra thin] (-0.1,  0.0) -- (\figxmax, 0.0) node[below]{$|\mathbb{D}\bv|$};
        \draw[ultra thin] ( 0.0, -0.1) -- (0.0, 1.1) node[left ]{$|\mathbb{S}|$};
        \draw[variable=\t,smooth,ultra thick]
          plot[domain=0:\figm] ({\t},{\fignua*\t})
          --
          plot[domain=\figm:\figmax] ({\t},{\fignua*\t+\fignub*(\t-\figm)*\t^(\figr-2)});
        \draw[ultra thin] (\figm,-0.03) node[below]{$m$} -- (\figm,0.03);
      \end{tikzpicture}
 \qquad &      \qquad   \begin{tikzpicture}[scale=2.2]%
        \def\figm{0.8}
        \def\figM{1.08}
        \def\figa{1.5}
        \def\figmax{1.02}
        \def\figxmax{1.8}
        \footnotesize
        \draw[ultra thin] (-0.1,  0.0) -- (\figxmax, 0.0) node[below]{$|\mathbb{D}\bv|$};
        \draw[ultra thin] ( 0.0, -0.1) -- (0.0, 1.1) node[left ]{$|\mathbb{S}|$};
        \draw[variable=\t,smooth,ultra thick]
          plot[domain=0:\figm] ({\t},{0})
          --
          plot[domain=\figm:\figmax] ({\t},{(\t-\figm)/(\figM^\figa-\t^\figa)^(1/\figa)});
        \draw[ultra thin] (\figm,-0.06) node[below]{$m$} -- (\figm,0);
        \draw[ultra thin,densely dashed] (\figM, -0.06) node[below]{$M$} -- (\figM, 1.1);
      \end{tikzpicture}
    \qquad &      \qquad   \begin{tikzpicture}[scale=2.2]%
        \def\fignua{0.4}
        \def\figm{0.55}
        \def\figmbar{1.2}
        \def\figM{1.48}
        \def\figa{1.6}
        \def\figmax{1.42}
        \def\figxmax{1.8}
        \footnotesize
        \draw[ultra thin] (-0.1,  0.0) -- (\figxmax, 0.0) node[below]{$|\mathbb{D}\bv|$};
        \draw[ultra thin] ( 0.0, -0.1) -- (0.0, 1.1) node[left ]{$|\mathbb{S}|$};
        \draw[variable=\t,smooth,ultra thick]
          plot[domain=0:\figm] ({\t},{0})
          --
          plot[domain=\figm:\figmbar] ({\t},{(\t-\figm)*\fignua})
          --
          plot[domain=\figmbar:\figmax] ({\t},{(\figmbar-\figm)*\fignua + (\t-\figmbar)/(\figM^\figa-\t^\figa)^(1/\figa)});
        \draw[ultra thin] (\figm,-0.06) node[below]{$\underline{m}$} -- (\figm,0);
        \draw[ultra thin] (\figmbar,-0.06) node[below]{$m$} -- (\figmbar,0);
        \draw[ultra thin,densely dashed] (\figM, -0.06) node[below]{$M$} -- (\figM, 1.1);
      \end{tikzpicture}
    \end{tabular}
    \caption{Sketches of the constitutive equations. The figure  on the left characterizes the activated Navier--Stokes fluids described in \eqref{zv_16a}. The  middle figure reflects the response of the activated Euler fluids described in \eqref{zv_17}. The figure  on the right vizualizes the Euler fluid with two activations described in \eqref{zv_1717}.}
    \label{zv_14}
\end{figure}
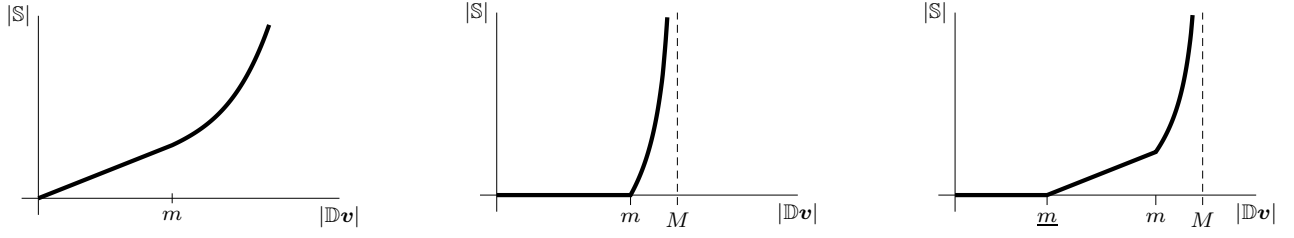

\section{Formulation of the problem and of the main result} \label{sec:2}

Let $L$, $T$, $\sigma$, $a$, $m$ and $M$ be fixed positive parameters (constants); while $L$, $T$ and $\sigma$ are arbitrary and $L$ specifies the periodic cell $\Omega:=(0,L)^3$ and $T$ a corresponding time-space cylinder $Q:=(0,T)\times\Omega$. The parameters $m$ and $M$ are assumed to satisfy
\begin{align}
   0<m<M-2<\infty. \label{ox1}
\end{align}
 A further restriction,  in terms of an upper bound on  the parameter $a$, will be stated later. Note that it would be sufficient to require in \eqref{ox1} that $0<m<M<\infty$, but then some minor technical issues regarding the construction of the approximation problems  to be considered have to be  dealt with more carefully.

With this set of parameters, we consider the following problem: for an arbitrarily  given \mbox{$\Omega$-periodic} function $\bv_0:\R^3\to\R^3$, ind an $\Omega$-periodic function $\bv=(v_1, v_2, v_3):(0,T)\times \R^3\to\R^3$, and a scalar function $p:[0,T]\times \R^3\to\R$ with $\Omega$-periodic gradient such that
\begin{subequations}
\label{zadani_classic}
\begin{align}
    \label{rovnice_classic}
    \di \bv = 0, \qquad \partial_t \bv &+ [\nabla\bv]\bv - \di  \mathbb{S}(\mathbb{D}\bv) = -\nabla p &&\textrm{in }Q,\\
    \label{constit_classic}
    \mathbb{S}(\mathbb{D}\bv)&:=\left\{\begin{aligned}
        &\mathbb{O} &&\textrm{ if } 0\le |\mathbb{D}\bv| \le m, \\
        &\sigma\frac{(|\D\bv| - m)}{(M^a - |\D\bv|^a)^{\frac 1a}}\frac{\D\bv}{|\D\bv|} &&\textrm{ if } m< |\mathbb{D}\bv| < M
    \end{aligned}\right. &&\textrm{in }Q,
        \\
    \label{initial_classic}
\bv(0, \cdot)&=\bv_0&& \textrm{in }\Omega.
\end{align}
\end{subequations}
Here, $\partial_t\bv$ and $\partial_{x_k}\bv$ denote the partial derivative of $\bv$ with respect to time $t$ and the spatial variable $x_k$; the operator $\nabla = (\partial_{x_1}, \partial_{x_2}, \partial_{x_3})$ stands for the gradient and the $i$-th component of $\di \mathbb{S}$ is specified by the relation $(\di \mathbb{S})_{i}:=\sum_{j=1}^3 \partial_{x_j} \mathbb{S}_{ij}$;  as in the previous section, the symbol
$\mathbb{D}\vect{v}:= \tfrac12 \left( \nabla \vect{v} + (\nabla \vect{v})^{\mathrm{T}} \right)$ denotes the symmetric part of the velocity gradient. 


\medskip

\noindent The main result of this paper is the following.

\begin{mdframed}[%
linecolor=white,leftmargin=60,%
rightmargin=60,
backgroundcolor=gray!30,%
innertopmargin=5pt,%
ntheorem]

\emph{Assume that $L$, $T$, $\sigma$, $a$, $m$ and $M$  are arbitrary positive parameters satisfying \eqref{ox1} and $0< a < 1/2$. Then, for  an arbitrary initial velocity $\bv_0$ satisfying
$$
\sup_{x\in\Omega} |\D\bv_0(x)|<M,
$$
there exists a unique weak solution $\bv$ to problem~\eqref{zadani_classic}.}

\end{mdframed}

\medskip

Before giving a precise definition of the weak solution to \eqref{zadani_classic}, we fix the notation  for the appropriate function spaces. Since we are dealing with a spatially periodic problem, the Sobolev spaces of $\Omega$-periodic  functions are defined as follows:
\begin{equation*}\label{periodic_def}
W_{per}^{k,p}(\Omega):=\overline{\left\{u=\tilde{u}_{\big|\Omega},\,\tilde{u}\in C^{\infty}(\R^3)\text{~is~}\Omega\text{-periodic} \right\}}^{\Vert\cdot\Vert_{k,p}},
\end{equation*}
where $k\in \mathbb{N}_0$ and $p\in[1,\infty)$ are arbitrary. Note that $L^p(\Omega) = W^{0,p}_{per}(\Omega)$. The space $W^{k,\infty}_{per}$ is then defined as
$$
W^{k,\infty}_{per}(\Omega):=W^{k,2}_{per}(\Omega)\cap W^{k,\infty}(\Omega).
$$

Sobolev spaces of vector-valued divergence-free $\Omega$-periodic functions are then defined in the following way:
$$
    W^{k,p}_{per,div}:=\{\vect{u}=(u_1, u_2, u_3);~ \di\vect{u} = 0 \textrm{ and } u_i\in W_{per}^{k,p}(\Omega) \textrm{ for } i=1,2,3\}.
$$
We use the symbols $\|\cdot\|_p$ and $\|\cdot\|_{k,p}$ to denote the usual norms in Lebesgue and Sobolev spaces.

Bold letters, e.g. $\vect{u}$, and blackboard bold letters, e.g. $\mathbb{S}$, are used for vector-valued and tensor-valued functions,  respectively, to distinguish them from scalar functions. If  two vector-valued functions $\vect{f}$, $\vect{g}$ are such that $\vect{f}\cdot\vect{g}\in L^1(\Omega)$, then $(\vect{f},\vect{g}):=\int_{\Omega} \vect{f}\cdot\vect{g}$; similarly, if two tensor-valued functions $\mathbb{S}$, $\mathbb{D}$ are such that $\mathbb{S}\cdot\mathbb{D}\in L^1(\Omega)$, then $(\mathbb{S},\mathbb{D}):=\int_{\Omega} \mathbb{S}\cdot\mathbb{D}$.

The shortcut ``a.e." abbreviates \emph{almost everywhere} and ``a.a." stands for \emph{almost all}. The symbols $|x|$, $x_{+}$, and $x_{-}$ denote the absolute value, the positive part, and the negative part of a real number $x$ so that $x=x_{+} - x_{-}$ and $|x|= x_{+}+x_{-}$. If  a set $A\subset \mathbb{R}^3$ is measurable, then $|A|$ stands for the Lebesgue measure of $A$.
Finally, if $\mathbb{D}=(D_{ij})_{i,j=1}^3$, then $|\mathbb{D}|:=\left(\sum_{i,j=1}^3 D_{ij}^2\right)^{1/2}$. Clearly, the symbol $|\cdot|$ has several different meanings, but its  particular choice should be clear from the  specific context.

Next, we define the notion of a \textit{weak solution} to~\eqref{zadani_classic} and formulate  our main result.
\begin{df}\label{weak_formulation}
Let $L$, $T$, $\sigma$, $a$, $m$ and $M$ be fixed positive parameters satisfying \eqref{ox1},  and let $\bv_0\in L^2_{ per,div}$. We say that $\bv$ is a weak solution to problem~\eqref{zadani_classic} if
\begin{align}
\bv&\in W^{1,2}\left(0,T; L^2(\Omega)\right)\cap L^2\left(0,T; W^{1,2}_{per,div}\right), \label{ox_ws1}\\
\mathbb{S}(\mathbb{D}\bv)&\in L^1(Q), \label{ox_ws2}\\
\|\mathbb{D}\bv\|_{\infty} &< M, \label{ox_ws3}
\end{align}
and
\begin{subequations}\label{flux-integrable_weak_formulation}
\begin{align}
       \begin{split}
       (\partial_t \bv, \vect{\varphi}) &-( \bv\otimes\bv, \mathbb{D}\vect{\varphi}) + (\mathbb{S}(\mathbb{D}\bv) ,\mathbb{D} \vect{\varphi}) =0 \\
       &\qquad\qquad \textrm{for all } \vect{\varphi}\in W^{1,2}_{per,div} \textrm{ with } \mathbb{D}\vect{\varphi}\in L^{\infty}(Q) \textrm{ and a.a. } t\in (0,T),
    \end{split}\label{rovnice}
    \\
    \mathbb{S}(\mathbb{D}\bv)&= \sigma\frac{(|\D\bv| - m)_{+}}{(M^a - |\D\bv|^a)^{\frac 1a}}\frac{\D\bv}{|\D\bv|} \textrm{ a.e. in }Q,\label{constit_old}
    \\
     \lim_{t\to 0_{+}} \Vert \bv(t,\cdot) &-\bv_0\Vert_2=0. \label{initial}
\end{align}
\end{subequations}
\end{df}
Note that it follows from \eqref{ox_ws1}--\eqref{ox_ws3} that $\bv$ is an admissible test function in \eqref{rovnice}. Consequently, the energy equality \eqref{zv_15}, with $\mathbb{S}$ given by \eqref{constit_old}, holds.

\begin{thm}\label{main_thm}
Let $L$, $T$, $\sigma$, $a$, $m$ and $M$ be fixed positive parameters satisfying \eqref{ox1} and $0<a<1/2$. Assume that $\bv_0\in W_{per,div}^{1,2}$ satisfies
\begin{equation}\label{flux_small_data}
    \|\D\bv_0\|_{\infty}<m.
\end{equation}
Then, there exists a unique weak solution to problem~\eqref{zadani_classic} in the sense of Definition \ref{weak_formulation} additionally satisfying
\begin{align}
\mathbb{S}(\mathbb{D}\bv)&\in L^{2(1-a)}(Q).\label{ox4}
\end{align}
\end{thm}

For clarity, we restrict ourselves to the three-dimensional setting, although the results  hold in any  number of space dimensions $d \geq 2$  (for $a\in (0,2/(d+1))$), including planar flows $d=2$. Also, only minor modifications give the same result if one replaces \eqref{constit_classic} by the constitutive relation \eqref{zv_1717}. As our main goal is to show that Euler fluids can be activated in such a way that the resulting model generates well-posed problems, we do not consider \eqref{zv_1717} in our analysis.

The proof is structured in the following way. In the next section,  Section ~\ref{sec:3}, we establish the  asserted uniqueness result. The identity that is satisfied by the $L^2$-norm of the difference of two  weak solutions, $\bv^1$ and $\bv^2$,  that stem from two different initial values, $\bv_0^1$ and $\bv_0^2$, motivates the requirement that the symmetric part of the velocity gradient of the solution  be a~priori bounded. Sections~\ref{sec:4} and \ref{sec:5} concern the  asserted existence result. In Section~\ref{sec:4}, we introduce a suitable $\varepsilon$-approximation of the problem~\eqref{zadani_classic},  which is then treated by the standard Faedo--Galerkin method in combination with a cascade of energy estimates that  help to establish the existence of a weak solution to the $\varepsilon$-approximation for arbitrary fixed $\varepsilon\in (0,1)$.
Finally, we derive and summarize the whole cascade of estimates that are uniform with respect to $\varepsilon$. Then, in Section \ref{sec:5}, letting $\varepsilon\to 0_{+}$, we identify a weak solution of the original problem.

The article ends with some concluding remarks, the bibliography and  an Appendix containing some properties of the function $\mathbb{S}(\cdot)$  that characterizes the response of the  fluid, and its approximations $\mathbb{S}_n(\cdot)$ that are used in the proof of the main result.

\section{Proof of uniqueness}\label{sec:3}

\noindent In this short section, we shall prove that there is at most one weak solution to the problem \eqref{zadani_classic}.

Let us assume that there are two weak solutions $\bv_1$ and $\bv_2$ to Problem \eqref{zadani_classic} with the same initial value $\bv_0\in L^2(\Omega)$. Both weak solutions are meant in the sense of Definition~\ref{weak_formulation}. Note that $\bv_1$ and $\bv_2$ are admissible test functions in \eqref{rovnice}, having a time derivative in $L^2(Q)$. Subtracting \eqref{rovnice} for $\bv_2$ from the same equation for $\bv_1$ and taking $\vect{\varphi}=\bv_1(t,\cdot)-\bv_2(t,\cdot)$ as a test function, we obtain (for a.a. for a.a. $t\in(0,T)$)
\begin{equation}\label{dec:4}
\frac12 \ddt \|\bv_1-\bv_2\|_{2}^2 + \int_{\Omega} \left(\mathbb{S}(\mathbb{D}\bv_1)-\mathbb{S}(\mathbb{D}\bv_2)\right)\cdot \left(\mathbb{D}\bv_1-\mathbb{D}\bv_2\right) \dd x = \int_{\Omega} (\bv_1-\bv_2)\otimes(\bv_1-\bv_2)\cdot \mathbb{D}\bv_2\,.
\end{equation}
The second term in \eqref{dec:4}, representing a monotone operator (see Lemma~\ref{monotone} in Appendix~\ref{Appendix}), is non-negative and we neglect it (as it cannot help in controlling the right-hand side of \eqref{dec:4}). Using however \eqref{ox_ws3} we conclude that
\begin{equation}\label{ox5}
   \frac12 \ddt \|\bv_1-\bv_2\|_{2}^2 \le M \|\bv_1-\bv_2\|_{2}^2.
\end{equation}
Setting $y(t):=\|\bv_1(t,\cdot)-\bv_2(t,\cdot)\|_{2}^2$, we conclude from \eqref{ox5} that
$$
y(t) \le e^{2Mt} y(0).
$$
As $y(0)=0$, this completes the proof of uniqueness. Note that the same argument gives the stability of the solution with respect to (small) perturbations of the initial velocity. Hence, the well-possedness of the evolution problem \eqref{zadani_classic} will be completed once we prove that it admits global-in-time existence of  a weak solution in the sense of Definition~\ref{weak_formulation}. The part of the proof asserting the existence of a weak solution constitutes the rest of the paper.

\section{\texorpdfstring{$\varepsilon$}{e}-approximations and their properties}\label{sec:4}

\noindent In this section, we introduce, for any $\varepsilon\in (0,1)$, an $\varepsilon$-approximation of the problem \eqref{zadani_classic} and show, by means of the  Faedo--Galerkin method and regularity techniques performed at the Galerkin level, that this  $\varepsilon$-approximation admits a unique weak solution with second spatial and first  time derivatives   contained in $L^2(Q)$.

For $\varepsilon\in (0,1)$, the $\varepsilon$-approximation of the problem \eqref{zadani_classic}  considered consists of adding a regularizing non-linear elliptic operator to  the balance of linear momentum  equation and regularizing the initial velocity $\vect{v}_0$ by  means of a standard mollifier\footnote{For $\vect{u}\in L^{2}_{div, per}$ (or smoother), we set $\vect{u}_{\varepsilon}:=\vect{u}*\omega_{\varepsilon}:=\int_{\mathbb{R}^3} \omega_{\varepsilon}(\cdot - y) \vect{u}(y) dy$ where $\omega_{\varepsilon}(x)=\varepsilon^{-3}\omega({|x|/\varepsilon})$ and $\omega$ is a non-negative smooth real-valued function having the compact support in $(-1,1)$ satisfying $\int_{\mathbb{R}} \omega = 1$.} denoted $\bv_{0\varepsilon}$.

Let $a>0$. We say that a couple of functions $(\bv,p)=(\bv^{\varepsilon}, p^{\varepsilon})$ solves the \mbox{$\varepsilon$-approximation} of the original problem \eqref{zadani_classic} if
\begin{subequations}\label{zadani_eps}
\begin{align}
    \label{rovnice_classic_eps}
    \di \bv = 0, \qquad \partial_t \bv + \di (\bv \otimes \bv) - \di  \mathbb{S} &= -\nabla p + \varepsilon \di \left((1+|\mathbb{D}\bv|^2)\mathbb{D}\bv\right) &&\textrm{in }Q,\\
    \label{constit_classic_eps}
    \mathbb{S}={\mathbb{S}}(\mathbb{D}\bv) &:=\sigma\frac{(|\D\bv| - m)_{+}}{(M^a - |\D\bv|^a)^{\frac 1a}}\frac{\D\bv}{|\D\bv|} &&\textrm{in }Q,\\
    \label{initial_classic_eps}
\bv(0, \cdot)&=\bv_{0\varepsilon}&& \textrm{in }\Omega,
\end{align}
\end{subequations}
and $\bv$, $\mathbb{S}$ and $\nabla p$ are $\Omega$-periodic. In order to  abbreviate the notation,   without loss of generality we set $\sigma = 1$ in what follows.

In accordance with the assumptions of Theorem \ref{main_thm}, we assume that  $\bv_0 \in W^{1,2}_{per,div}$ satisfies $\|\mathbb{D}\bv_0\|_{\infty} < m$.  Thanks to the spatial periodicity of the studied problem and the properties of the mollifier, we observe that, for any $p\in(1,\infty)$,  one has that
\begin{equation}
    \label{kk1} \|\mathbb{D}\bv_{0\varepsilon}\|_p = \|(\mathbb{D}\bv_{0})_{\varepsilon}\|_p \le \|\mathbb{D}\bv_{0}\|_p.
\end{equation}
Consequently,
\begin{equation}
    \label{kk2} \|\mathbb{D}\bv_{0\varepsilon}\|_{\infty} =\lim_{p\to\infty} \|\mathbb{D}\bv_{0\varepsilon}\|_p \le  \lim_{p\to\infty}\|\mathbb{D}\bv_{0}\|_p = \|\mathbb{D}\bv_0\|_\infty < m,
\end{equation}
and $\bv_{0\varepsilon}\in W^{k,2}_{per,div}$, $k\in \mathbb{N}$,  satisfies the same bound as $\bv_0$, namely $\|\mathbb{D}\bv_{0\varepsilon}\|_{\infty}<m$.

We say that $\bv = \bv^{\varepsilon}$ is a \emph{weak solution to \eqref{zadani_eps}} if
\begin{align}\label{dec:34}
&\bv \in L^2\left(0,T; W^{2,2}_{per} \cap W^{1,2}_{per,div} \right),\qquad \partial_t \bv \in L^2\left(Q\right), \qquad \|\mathbb{D}\bv\|_{\infty, Q}\le M, \qquad \mathbb{S}(\D\bv)\in L^1(Q),
\\
&\lim_{t\to 0_{+}}\Vert \bv(t,\cdot)-\bv_{0\varepsilon} \Vert_2=0,\label{wf-in}
\end{align}
and
\begin{equation}
    \begin{split}
       (\partial_t \bv,\vect{\varphi}) - (\bv\otimes\bv,\mathbb{D}\vect{\varphi}) + ({\mathbb{S}}(\mathbb{D}\bv), \mathbb{D}\vect{\varphi}) &+ \varepsilon ((1+|\mathbb{D}\bv|^2) \mathbb{D}\bv,\mathbb{D}\vect{\varphi})= 0 \\ &\qquad \textrm{ for all } \vect{\varphi}\in W^{1,\infty}_{per,div} \textrm{ and a.a. } t\in (0,T).
    \end{split}\label{wf-rovnice}
\end{equation}
Uniqueness of such a solution follows from the same argument as in Section \ref{sec:3}. To establish the existence of the solution, we apply the  Faedo--Galerkin method combined with higher differentiability estimates that we will  derive at the level of Galerkin approximations. These estimates and the limit from the Galerkin approximation to the continuous level represent the core of this section. 

\subsection{Galerkin approximations}\label{4.1} Consider the basis $\left\{\vect{\omega}_r\right\}_{r=1}^{\infty}$ in $W_{per,div}^{1,2}$ consisting of solutions of the following spectral problem:
\begin{align}\label{Dec:GS}
    \int_{\Omega}\nabla\vect{\omega}_r\cdot\nabla\vect{\varphi}\dd x&=\lambda_r\int_{\Omega}\vect{\omega}_r\cdot\vect{\varphi}\dd x \qquad\text{~for all~} \vect{\varphi}\in W^{1,2}_{per,div}.
\end{align}
It is well known (see, e.g. \cite{temam} or \cite[Appendix A.4]{mnrr96}) that there is a nondecreasing sequence of (positive) eigenvalues $\{\lambda_r\}_{r=1}^{\infty}$ and a corresponding set of eigenfunctions $\left\{\vect{\omega}_r\right\}_{r=1}^{\infty}$ that are orthogonal in  $W^{1,2}_{per}$ and orthonormal in  $L^2_{per}$. Moreover, the projections $\mathcal{P}^n$ defined through $\mathcal{P}^n(\bv)=\sum_{i=1}^n (\bv, \vect{\omega}_i) \vect{\omega}_i$ are continuous as mappings from $L^2_{per,div}$ to $L^2_{per,div}$ and from $W^{1,2}_{per,div}$ to $W^{1,2}_{per,div}$. Also,   thanks to $\Omega$-periodicity and elliptic regularity, the $\Omega$-periodic extensions of $\vect{\omega}_r$ belong to $C^{\infty}(\R^3)$, and the projections $\mathcal{P}^n$ are continuous from $W^{k,2}_{per}$ to $W^{k,2}_{per}$ for any $k\in \mathbb{N}$. In fact,
\begin{equation}
    \label{kk3} \|P^n\|_{\mathcal{L}(L^2_{per,div})} = 1, \quad \|P^n\|_{\mathcal{L}(W^{1,2}_{per,div})} = 1 \quad\textrm{ and }\quad \|P^n\|_{\mathcal{L}(W^{k,2}_{per,div})} = C_k,
\end{equation}
where $C_k$ are certain positive constants.
 These results, in conjunction with the interpolation inequality $\|z\|_{1,\infty}\le C \|z\|_{1,2}^{1/4} \|z\|_{3,2}^{3/4}$, imply that (for any $\vect{u}\in W^{3,2}_{per}$)
\begin{equation}
    \label{ka_1}
    \|P^n(\mathbb{D}\vect{u}) - \mathbb{D}\vect{u}\|_{\infty}\le
    C \|P^n\vect{u} - \vect{u}\|_{1,2}^{1/4} \|P^n\vect{u} - \vect{u}\|_{3,2}^{3/4} \le C C^{3/4}_3 \|\vect{u}\|_{3,2}^{3/4}\|P^n\vect{u} - \vect{u}\|_{1,2}^{1/4} \to 0 \textrm{ as } n\to \infty.
\end{equation}
Consequently, using \eqref{kk2}, we observe that, for $n$ sufficiently large,
\begin{equation}
    \label{kk4}
    \|P^n(\mathbb{D}\bv_{0\varepsilon})\|_{\infty} \le \|P^n(\mathbb{D}\bv_{0\varepsilon}) - \mathbb{D}\bv_{0\varepsilon}\|_{\infty} + \|\mathbb{D}\bv_{0\varepsilon}\|_{\infty} < m.
\end{equation}

Next, for an arbitrary  but fixed $n\in \N$, we look for $\bv^n$ in the form
\begin{equation}\label{galerkin_def}\notag
    \bv^n(t,x)=\sum_{r=1}^nc_r^n(t)\,\vect{\omega}_r(x),
\end{equation}
where the coefficients  $c_r^n$, $r=1, \dots, n$, are determined as the solution of the following   ODE initial-value problem:
\begin{align}\label{galerkin1}
    (\partial_t{\bv}^n, \vect{\omega}_r) &+ ([\nabla \bv^n]\bv^n, \vect{\omega}_r) + ({\mathbb{S}}_n(\mathbb{D}\bv^n), \mathbb{D}\vect{\omega}_r) + \varepsilon\left((1+|\mathbb{D}\bv^n|^2) \mathbb{D}\bv^n,  \mathbb{D}\vect{\omega}_r\right) =0,\quad r=1,\ldots, n,\\ \bv^n(0, \cdot)&= \mathcal{P}^n(\bv_0) \quad \iff \quad c^n_r(0)= (\bv_0, \vect{\omega}_r), \qquad r=1,\dots,N,\label{galerkin2}
\end{align}
where
\begin{equation}
{\mathbb{S}}_n(\mathbb{D}) := f_n(|\mathbb{D}|)\frac{\mathbb{D}}{|\mathbb{D}|} := \frac{(|\mathbb{D}|-m)_{+}}{\left( M^a - (\min\{|\mathbb{D}|,M-1/n\})^{a}\right)^{\frac{1}{a}}}\frac{\mathbb{D}}{|\mathbb{D}|}. \label{galerkin3}
\end{equation}
Note that
\begin{equation}
    |{\mathbb{S}}_n(\mathbb{D})| = f_n(|\mathbb{D}|) \ge 0. \label{galerkin4}
\end{equation}
 Further details concerning the properties of $\mathbb{S}$ and ${\mathbb{S}}_n$ are given in  the Appendix. We will not recall them here, but  will always refer to the Appendix before we use them.

 By defining $\vect{c}^n:=(c_1^n, c_2^n, \dots, c_n^n)$, the above   ODE initial-value problem \eqref{galerkin1}, \eqref{galerkin2} can be written in the form
$$
   \frac{\textrm{d}\vect{c}^n}{\textrm{d} t} = \vect{g}_n(\vect{c}^n), \qquad \vect{c}^n(0) =P^n\bv_0.
$$
As $\vect{g}_n$ is a continuous function, the local-in-time existence of the solution to \eqref{galerkin1}-\eqref{galerkin3} follows from Peano's  existence theorem. The first uniform  estimate that we establish next, see \eqref{est_2} below, then implies the existence of the solution on the whole interval $[0,T]$. The properties of the solution established below  guarantee also its uniqueness (  by arguing similarly as in Sect.~\ref{sec:3}).

\subsection{First uniform estimates}\label{4.2}
Multiplying the $r$-th equation in~\eqref{galerkin1} by $c_r$, summing these equations  through $r=1,\dots,n$, and integrating the result over  the time interval $(0,t)$ where $t\in (0,T]$, we obtain
\begin{equation}\label{est_1}
    \frac 12 \Abs{\bv^n(t, \cdot)}_{2}^2 + \int_0^t ({\mathbb{S}}_n(\mathbb{D}\bv^n), \mathbb{D}\bv^n) + \varepsilon \int_0^t \left((1+|\mathbb{D}\bv^n|^2) \mathbb{D}\bv^n,  \mathbb{D}\bv^n\right) = \frac12 \Abs{P^n\bv_{0\varepsilon}}_{2}^2 \le \frac12 \Abs{\bv_{0\varepsilon}}_{2}^2 \le \frac12 \Abs{\bv_{0}}_{2}^2.
\end{equation}
This leads to
\begin{align}\label{est_2}
    \sup_{t\in (0,T)} |\vect{c}^n(t)|_{\mathbb{R}^n} = \sup_{t\in (0,T)} \Abs{\bv^n(t, \cdot)}_{2}^2 &\le \Abs{\bv_0}_{2}^2, \\
    \int_0^T ({\mathbb{S}}_n (\mathbb{D}\bv^n), \mathbb{D}\bv^n) + \varepsilon \int_0^T \left((1+|\mathbb{D}\bv^n|^2)\mathbb{D}\bv^n,  \mathbb{D}\bv^n\right) &\le \Abs{\bv_0}_{2}^2. \label{est_25}
\end{align}
Using also \eqref{galerkin3}, we conclude from \eqref{est_25} that
\begin{align}
    \int_Q f_n (|\mathbb{D}\bv^n|) |\mathbb{D}\bv^n| + \varepsilon \int_0^T \Abs{\mathbb{D}\bv^n}_{4}^{4} &\le \Abs{\bv_0}_{2}^2.\label{est_3}
\end{align}
Since $f_n$ vanishes on $\{(t,x)\in Q; |\mathbb{D}\bv^n|\le m\}$, it follows from \eqref{est_3} that
\begin{equation}
    \label{est_4}
    m \int_Q f_n(|\mathbb{D}\bv^n|) \le \Abs{\bv_0}_{2}^2,
\end{equation}
which in combination with \eqref{galerkin4} implies that
\begin{equation}
    \label{est_5}
    \int_Q |\mathbb{S}^n| \le \frac{\Abs{\bv_0}_{2}^2}{m}, \qquad \textrm{ where } \quad \mathbb{S}^n:= {\mathbb{S}}_n(\mathbb{D}\bv^n).
\end{equation}
The estimates \eqref{est_2}, \eqref{est_25}, \eqref{est_3} and  \eqref{est_5} are uniform with respect to $n$ and $\varepsilon$.

We also conclude from \eqref{est_3} and \eqref{galerkin3} that
\begin{align*}
    \|\bv_0\|^2_2 \ge \int_{\left\{|\mathbb{D}\bv^n|\ge M-\tfrac{1}{n}\right\}} f_n (|\mathbb{D}\bv^n|) |\mathbb{D}\bv^n| \ge \frac{(M-m-\tfrac{1}{n})(M-\tfrac{1}{n})}{(M^a - (M-\tfrac{1}{n})^a)^{\frac{1}{a}}}\,\, \left|\left\{|\mathbb{D}\bv^n|\ge M-\tfrac{1}{n}\right\}\right|,
\end{align*}
which leads to
\begin{align} \label{est_5b}
    \left|\left\{|\mathbb{D}\bv^n|\ge M-\tfrac{1}{n}\right\}\right|\le \|\bv_0\|^2_2 \frac{(M^a - (M-\tfrac{1}{n})^a)^{\frac{1}{a}}}{(M-m-\tfrac{1}{n})(M-\tfrac{1}{n})} \to 0 \quad \textrm{ as } n\to \infty\,.
\end{align}
Since, for any $\omega >0$, the function $s\mapsto s/(s+\omega)$ is nonnegative and bounded from above by $1$ on $[0, +\infty)$, it follows from \eqref{est_5b} that
\begin{equation}\label{est_5c}
    \lim_{n\to \infty} \int_Q \frac{\left(|\mathbb{D}\bv^n| - M-\tfrac{1}{n}\right)_+}{\left(|\mathbb{D}\bv^n| - M-\tfrac{1}{n}\right)_+ + \omega} = 0.  \qquad (\textrm{for any } \omega>0)
\end{equation}

We also observe from \eqref{est_3} that
\begin{equation}
    \label{est_6} \int_0^T \Abs{\mathbb{D}\bv^n}_{4}^{4} \le C(\varepsilon^{-1}),
\end{equation}
where $C(\varepsilon^{-1}):=\frac{\Abs{\bv_0}_{2}^2}{\varepsilon}$. In what follows, $C(\varepsilon^{-1})$ denotes a function of $\varepsilon$ that is independent of $n$ but tends to $\infty$ as $\varepsilon\to 0_{+}$; its particular form can change from line to line.

\subsection{Spatial derivative estimates}\label{4.3}~This time, we multiply the $r$-th equation in \eqref{galerkin1} by $\lambda_r c_r$ and sum the identities obtained  through $r=1,\dots,N$. Using also \eqref{Dec:GS}, we get
\begin{align*}
\frac12 \ddt \|\nabla \bv^n\|_2^2 + (\nabla{\mathbb{S}}_n(\mathbb{D}\bv^n), \nabla\mathbb{D}\bv^n) &+ \varepsilon \left(\nabla \left[(1+|\mathbb{D}\bv^n|^2) \mathbb{D}\bv^n\right],  \nabla \mathbb{D}\bv^n\right)\\ &= - \int_{\Omega} \partial_{x_s} v^n_k \partial_{x_k} v^n_i \partial_{x_s} v^n_i.
\end{align*}
Using \eqref{ox_22a} and integrating the last identity w.r.t.~time over $(0,t)$, $t\in (0,T)$, we arrive at
\begin{equation}\label{E2}
\begin{split}
\frac12 \|\nabla \bv^n(t, \cdot)\|_2^2 + \int_{Q_t} \mathcal{Q}_n(\mathbb{D}\bv^n, &\nabla \mathbb{D}\bv^n) + \varepsilon \int_{Q_t} (1+|\mathbb{D}\bv^n|^2) |\nabla \mathbb{D}\bv^n|^2 + 2|\mathbb{D}\bv^n\cdot \mathbb{D}(\nabla\bv^n)|^2 \\ &= - \int_{Q_t} \partial_{x_s} v^n_k \partial_{x_k} v^n_i \partial_{x_s} v^n_i + \frac12 \Abs{P^n\nabla\bv_{0\varepsilon}}_{2}^2 \\
&\le - \int_{Q_t} \partial_{x_s} v^n_k \partial_{x_k} v^n_i \partial_{x_s} v^n_i + \frac12 \Abs{\nabla\bv_0}_{2}^2.
\end{split}
\end{equation}
Since
$$
\left| \int_{Q} \partial_{x_s} v^n_k \partial_{x_k} v^n_i \partial_{x_s} v^n_k \right| \le \int_0^T \|\nabla \bv^n\|_3^3,
$$
we conclude from \eqref{E2} and \eqref{est_6} that
\begin{equation}
    \sup_{t\in (0,T)} \|\nabla \bv^n(t, \cdot)\|_2^2 + \int_{0}^T \|\bv^n\|_{2,2}^2 + \int_{Q} \mathcal{Q}_n(\mathbb{D}\bv^n, \nabla \mathbb{D}\bv^n) + \int_0^T \|\nabla |\mathbb{D}\bv^n|^2\|_2^2 \le C(\varepsilon^{-1}). \label{est_7}
\end{equation}

\subsection{Time derivative estimate (uniform with respect to \texorpdfstring{$n$}{n})}\label{4.4}
We multiply the $r$-th equation in~\eqref{galerkin1} by $\ddt c_r$ and sum these equations  through $r=1,\dots,N$. This leads to
\begin{equation}\label{zvole_1}
        \|\partial_t \bv^n\|_2^2 + \varepsilon \ddt \int_{\Omega} \frac{|\mathbb{D}\bv^n|^2}{2} + \frac{|\mathbb{D}\bv^n|^4}{4} + \ddt \int_{\Omega} F_n(|\mathbb{D}\bv^n|)
        = - 2\int_{\Omega} \bv^n \cdot (\mathbb{D}\bv^n)\partial_t \bv^n =:I_n,
\end{equation}
where
\begin{equation*}
    F_n(|\mathbb{D}|):= \int_0^{|\mathbb{D}|}\frac{(s-m)_{+}}{\left(M^a - (\min\{s,M-1/n\})^{a}\right)^{\frac{1}{a}}}\, \dd s.
\end{equation*}
Note that
\begin{equation}
   F_n(|\mathbb{D}|) = 0 \qquad \textrm{ if } \quad |\mathbb{D}|\le m \label{kk5}
\end{equation}
and
$$
|F_n(|\mathbb{D}_1|) - F_n(|\mathbb{D}_2|)| \le 2^{-1}\left[M^a - (M-1/n)^a\right]^{-\frac{1}{a}}((|\mathbb{D}_1|-m)^2-(|\mathbb{D}_2|-m)^2) \qquad \textrm{ otherwise}.
$$

After integrating \eqref{zvole_1} over $(0,t)$  for $t\in(0,T]$, we arrive at
\begin{equation}\label{zvole_2}
    \begin{split}
            \frac{\varepsilon}{2} \|\mathbb{D}\bv^n(t,\cdot)\|_2^2 + \frac{\varepsilon}{4} \|\mathbb{D}\bv^n(t,\cdot)\|_4^4 + \int_{\Omega} F_n(|\mathbb{D}\bv^n(t,\cdot)|) &+ \int_0^t \|\partial_t \bv^n\|_2^2
        = \int_0^t I_n \\ &+ \int_{\Omega} F_n(|P^n\mathbb{D}\bv_{0\varepsilon}|) + \frac{\varepsilon}{2} \|P^n\mathbb{D} \bv_{0\varepsilon}\|_2^2 + \frac{\varepsilon}{4} \|P^n\mathbb{D}\bv_{0\varepsilon}\|_4^4.
    \end{split}
\end{equation}
By virtue of \eqref{kk4} and \eqref{kk5}, $F_n(|P^n\mathbb{D}\bv_{0\varepsilon}|)=0$ for $n$ large enough, and the last two terms in \eqref{zvole_2} are bounded by the constant $C= C(m)$ depending on $m$ but independent of $n$ and $\varepsilon$. Next, we observe, using also \eqref{est_6}  in the last inequality, that
\begin{equation}
    \label{zvole_3}
    \begin{split}
    \int_0^t |I_n| \le 2\int_0^{t} \|\bv^n\|_4 \|\mathbb{D}\bv^n\|_4 \|\partial_t \bv^n\|_2 &\le C \int_0^{t} \|\bv^n\|_{1,4}^2\|\partial_t \bv^n\|_2 \\ &\le
    \frac{1}{2}\int_0^t \|\partial_t \bv^n\|_2^2 + \frac{C^2}{2} \int_0^t \|\bv^n\|_{1,4}^4 \le \frac{1}{2}\int_0^t \|\partial_t \bv^n\|_2^2 + C(\varepsilon^{-1}).
    \end{split}
\end{equation}
Combining \eqref{zvole_3} with \eqref{zvole_2}, using also Korn's inequality, we observe that
\begin{align}
    \int_0^T \|\partial_t \bv^n\|_{2}^2 &\le C(\varepsilon^{-1}), \label{zvole_4} \\ \sup_{t\in (0,T)} \|\nabla \bv^n\|_4 &\le C(\varepsilon^{-1}). \label{zvole_5}
\end{align}
\subsection{Additional time derivative estimate (uniform with respect to \texorpdfstring{$n$}{n})}\label{4.5}
We differentiate~\eqref{galerkin1} with respect to time and multiply the $r$-th equation of the result by $\ddt c_r$. After summing the resulting equations  through $r=1,\dots,N$, we obtain
\begin{equation}\label{k1}
\begin{split}
    \frac12 \ddt \|\partial_t \bv^n\|_2^2 + (\partial_t{\mathbb{S}}_n(\mathbb{D}\bv^n), \partial_t \mathbb{D}\bv^n) &+ \varepsilon \left(\partial_t \left[(1+|\mathbb{D}\bv^n|^2) \mathbb{D}\bv^n\right],  \partial_t \mathbb{D}\bv^n\right)\\ &= - \int_{\Omega} (\partial_t\bv^n \otimes \partial_t\bv^n)\cdot \mathbb{D}\bv^n,
\end{split}
\end{equation}
where we used the   following identity, valid for any regular enough $\vect{u}$ satisfying $\di\vect{u} = 0$:
\begin{align*}
    \int_{\Omega} \partial_t ( u_k \partial_{x_k} u_i) \partial_t u_i &= \int_{\Omega}   \partial_t u_k \partial_{x_k} u_i \partial_t u_i + \int_{\Omega} u_k \partial^{(2)}_{x_k, t}\vect{u} \cdot \partial_t \vect{u} = \int_{\Omega} (\partial_t\vect{u} \otimes \partial_t\vect{u})\cdot \nabla\vect{u} + \int_{\Omega} \vect{u} \cdot \nabla\left(\frac{|\partial_t \vect{u}|^2}{2}\right) \\ &=\int_{\Omega} (\partial_t\vect{u} \otimes \partial_t\vect{u})\cdot \mathbb{D}\vect{u}.
\end{align*}
Using the notation introduced in \eqref{ox_22} and the relation given in \eqref{ox_22a}, multiplying \eqref{k1} by a standard truncation function $\zeta\in C^{\infty}(\mathbb{R})$ satisfying $\zeta(0)=0$, $\zeta > 0$ in $(0, \infty)$ and $\zeta(t)=1$ on $[\delta,\infty)$ for any $\delta>0$, and integrating the result over $(0,t)$, we obtain
\begin{equation}\label{k2}
\begin{split}
        \frac12 \zeta(t) \|\partial_t \bv^n(t,\cdot)\|_2^2 &+ \int_{Q_t} \zeta \mathcal{Q}_n(\mathbb{D}\bv^n, \partial_t\mathbb{D}\bv^n) + \varepsilon \int_{Q_t} \zeta \left[ (1+|\mathbb{D}\bv^n|^2) |\mathbb{D}(\partial_t \bv^n)|^2 + 2|\mathbb{D}\bv^n\cdot \mathbb{D}(\partial_t\bv^n)|^2 \right] \\ &
        = \int_{Q_t} \zeta' |\partial_t \bv^n|^2 - \int_{Q_t} \zeta (\partial_t \bv^n\otimes\partial_t \bv^n)\cdot (\mathbb{D}\bv^n) =: J_n^1 + J_n^2.
\end{split}
\end{equation}
Next, for any $t\in (\delta, T]$, we estimate the terms $J_n^1$ and $J_n^2$ on the right-hand side of \eqref{k2} with the help of \eqref{zvole_4}, \eqref{zvole_5} and H\"{o}lder's  inequality and interpolation inequalities as follows:
\begin{align*}
    |J_n^1| &\le C(\delta^{-1}) \int_0^T \|\partial_t\bv^n\|_2^2 \le C(\delta^{-1}, \varepsilon^{-1}), \\
    |J_n^2| &\le \int_{Q_t} \zeta |\partial_t\bv^n|^2 |\mathbb{D}\bv^n| \le \int_0^t \zeta \|\partial_t\bv^n\|_{8/3}^2 \|\mathbb{D}\bv^n\|_4 \le C \!\!\sup_{t\in (0,T)}\|\nabla \bv^n\|_4 \int_0^t  \zeta \|\partial_t\bv^n\|_{2}^{5/4} \|\mathbb{D}(\partial_t\bv^n)\|_{2}^{3/4} \\
    &\le C(\varepsilon^{-1}) \int_0^t (\|\partial_t\bv^n\|_{2}^2)^{5/8} (\zeta \|\mathbb{D}(\partial_t\bv^n)\|_{2}^2)^{3/8} \le \frac{\varepsilon}{2} \int_{Q_t} \zeta |\mathbb{D}(\partial_t\bv^n)|^2 + C(\varepsilon^{-1}) \int_0^T \|\partial_t \bv^n\|_2^2 \\
    & \le \frac{\varepsilon}{2} \int_{Q_t} \zeta |\mathbb{D}(\partial_t\bv^n)|^2 + C(\varepsilon^{-1}).
\end{align*}
Incorporating these estimates into \eqref{k2} we conclude that, for any $T\ge t > \delta>0$,
\begin{equation}\label{k3}
\begin{split}
        \sup_{t\in [\delta, T]} \|\partial_t \bv^n(t,\cdot)\|_2^2 &+ \int_{Q_{\delta, T}} \mathcal{Q}_n(\mathbb{D}\bv^n, \partial_t\mathbb{D}\bv^n) + \varepsilon \int_{Q_{\delta, T}} |\mathbb{D}(\partial_t \bv^n)|^2 \le C(\delta^{-1}, \varepsilon^{-1}),
\end{split}
\end{equation}
where $Q_{\delta,t}:=(\delta,t)\times \Omega$, $t\in (\delta, T]$.

Next, we ``repeat" the whole procedure again, but replace the time derivative by a time difference. For this purpose, for any $t\in (0,T)$ and for any $\tau$ (positive or negative) small enough so that $t+\tau\in (0,T)$ and for any function $\vect{u}$, we set
\begin{equation}\label{kdef}
  \eth_{\tau}\vect{u}(t,x):= \vect{u}(t+\tau,x) - \vect{u}(t,x).
\end{equation}
Then we consider \eqref{galerkin1} at times $t+\tau$ and $t$ and take their difference. Multiplying the $r$-th equation of the result by $c_r(t+\tau)-c_r(t)$ and summing the resulting equations  through $r=1,\dots,N$, we arrive at
\begin{equation}\label{k1dif}
\begin{split}
    \frac12 \ddt \|\eth_{\tau} \bv^n\|_2^2 + (\eth_{\tau}{\mathbb{S}}_n(\mathbb{D}\bv^n), \eth_{\tau} \mathbb{D}\bv^n) &+ \varepsilon \left(\eth_{\tau} \left[(1+|\mathbb{D}\bv^n|^2) \mathbb{D}\bv^n\right],  \eth_{\tau} \mathbb{D}\bv^n\right)\\ &= - \int_{\Omega} (\eth_{\tau}\bv^n \otimes \eth_{\tau}\bv^n)\cdot \mathbb{D}\bv^n,
\end{split}
\end{equation}
where we used the   following identity (for $\vect{u}$ satisfying $\di\vect{u} = 0$):
\begin{align*}
    \int_{\Omega} \eth_{\tau} ( u_k \partial_{x_k} u_i) \eth_{\tau} u_i &= \int_{\Omega} u_k (\cdot + \tau, \cdot) \partial_{x_k} \eth_{\tau} \vect{u} \cdot \eth_{\tau} \vect{u} + \int_{\Omega}   \eth_{\tau} u_k \partial_{x_k} u_i \eth_{\tau} u_i \\ &= \int_{\Omega} \vect{u} \cdot \nabla\frac{|\eth_{\tau} \vect{u}|^2}{2} + \int_{\Omega} (\eth_{\tau}\vect{u} \otimes \eth_{\tau}\vect{u})\cdot \nabla\vect{u}   =\int_{\Omega} (\eth_{\tau}\vect{u} \otimes \eth_{\tau}\vect{u})\cdot \mathbb{D}\vect{u}.
\end{align*}
Multiplying \eqref{k1dif} by $\zeta$ introduced before \eqref{k2} and integrating the result over $(0,T)$, we get
\begin{equation}\label{k2dif}
\begin{split}
        \frac12 \|\eth_{\tau} \bv^n(T,\cdot)\|_2^2 &+ \int_{Q} \zeta (\eth_{\tau}{\mathbb{S}}_n(\mathbb{D}\bv^n), \eth_{\tau} \mathbb{D}\bv^n) + \varepsilon \int_{Q} \zeta \left(\eth_{\tau} \left[(1+|\mathbb{D}\bv^n|^2) \mathbb{D}\bv^n\right],  \eth_{\tau} \mathbb{D}\bv^n\right) \\ &
        = \int_{Q} \zeta' |\eth_{\tau} \bv^n|^2 - \int_{Q} \zeta (\eth_{\tau} \bv^n \otimes\eth_{\tau} \bv^n)\cdot (\mathbb{D}\bv^n).
\end{split}
\end{equation}

\subsection{Improved estimates for \texorpdfstring{$\mathbb{S}^n$}{S} introduced in \texorpdfstring{\eqref{est_5b}}{5}
} \label{4.6}
Referring to \eqref{est_5b} for the definition of $\mathbb{S}^n := \mathbb{S}_n(\mathbb{D}\bv^n)$ and the uniform  bound on $\mathbb{S}^n$ established in the non-reflexive space $L^1(Q)$, the aim of this subsection is to improve the  uniform bound on $\mathbb{S}^n$.  To this end, we first notice that,  because the function $f_n$ is increasing on $[m,M-1/n]$  (c.f.~\eqref{galerkin3}), there is  an $s_*\in (0,\infty)$ such that
\begin{equation}\label{k70}
|\mathbb{S}_n((m+M)/2)| = s_* \qquad \left(\textrm{i.e.,} \quad s_* = \frac{((M-m)/2)}{(M^a - ((M+m)/2)^a)^{1/a}}\right).
\end{equation}
Consequently, independently of $n$ and $\varepsilon$,
\begin{equation}\label{k71}
\|\mathbb{S}^n\|_{L^{\infty}(Q\cap\{|\mathbb{D}\bv^n|\le (M+m)/2)} \le C.
\end{equation}
In the rest of this subsection, we will concentrate on improving the uniform bounds on $\mathbb{S}^n$ on the set $Q\cap\{|\mathbb{D}\bv^n|\ge (M+m)/2\}$.
We first notice that it follows from \eqref{est_7} and \eqref{k3} that
\begin{equation}\label{k72}
\int_{Q_{\delta,T}} \mathcal{Q}_n(\mathbb{D}\bv^n, \nabla \mathbb{D}\bv^n) + \mathcal{Q}_n(\mathbb{D}\bv^n, \partial_t \mathbb{D}\bv^n) \le C(\delta^{-1},\varepsilon^{-1}).
\end{equation}
Using \eqref{ox_24}, and introducing the notation $\nabla_{t,x}:=(\partial_t, \nabla)=(\partial_t,\partial_{x_1},\partial_{x_2}, \partial_{x_3})$ we conclude that
\begin{equation}\label{k73}
\int_{Q_{\delta,T}\cap\{\{|\mathbb{D}\bv^n|\ge (M+m)/2\}} \frac{|\nabla_{t,x}\mathbb{S}^n|^2}{(1+|\mathbb{S}^n|)^{1+a}} \le C(\delta^{-1},\varepsilon^{-1}).
\end{equation}
Recalling the definition of $s_*$, see \eqref{k70}, it follows from \eqref{k73} that
$$
\int_{Q_{\delta,T}} \frac{|\nabla_{t,x}(|\mathbb{S}^n| - s_*)_{+}|^2}{\left[(1+ s_* + |\mathbb{S}^n| - s_*))^{\frac{1+a}{2}}\right]^2} \le C(\delta^{-1},\varepsilon^{-1}).
$$
Hence,
$$
\int_{Q_{\delta,T}} |\nabla_{t,x}\left(1+s_*+(|\mathbb{S}^n| - s_*)_{+}\right)^{\frac{1-a}{2}}|^2 \le C(\delta^{-1},\varepsilon^{-1}).
$$
Applying the Sobolev embedding of $W^{1,2}(Q)$ into $L^4(Q)$, we then conclude that
\begin{equation} \label{k74}
    \int_{Q_{\delta,T}} \left(1+s_* + (|\mathbb{S}^n| - s_*)_{+}\right)^{2(1-a)}\le C(\delta^{-1},\varepsilon^{-1}).
\end{equation}
This together with \eqref{k71} implies that, for $2(1-a)>1 \iff a\in (0,1/2)$, $\{\mathbb{S}^n\}$ satisfies the improved estimates (uniform with respect to $n$)
\begin{equation}\label{k75}
\|\mathbb{S}^n\|_{L^{2(1-a)}(Q_{\delta,T})} \le C(\delta^{-1},\varepsilon^{-1}).
\end{equation}

\subsection{Limit \texorpdfstring{$n\to \infty$}{n}} \label{4.7}~
 Thanks to the reflexivity and separability of the underlying function spaces and the Aubin-Lions compactness lemma, it follows from the estimates \eqref{est_6}, \eqref{est_7}, \eqref{zvole_4}, \eqref{zvole_5}, \eqref{k3} and \eqref{k75} that there is a subsequence of $\left\{(\bv^n, \mathbb{S}^n)\right\}_{n=1}^{\infty}$ (which we do not relabel) such that 
\begin{subequations}\label{limity_n}
\begin{align}
    \label{k84}
    \bv^n&\rightharpoonup \bv &&\text{~weakly in~} L^2\left(0,T; W^{2,2}_{per}\right) \textrm{ and *-weakly in~} L^{\infty}\big(0,T, W^{1,4}_{per,div}\big), \\ \label{k85}
    \partial_t{\bv}^n&\rightharpoonup \, \partial_t {\bv} &&\text{~weakly in~} L^{2}(Q), \\ \label{k86a}
    \bv^n&\to \bv &&\text{~strongly in~} L^q(Q), \qquad (q\in [1, \infty)), \\\label{k86}
    \mathbb{D}\bv^n&\to \mathbb{D}\bv &&\text{~strongly in~} L^q(Q), \qquad (q\in [1, 10/3)), \\ \label{k88}
    \mathbb{D}\bv^n&\to \mathbb{D}\bv &&\text{~a.e. in~} Q, \\
    \label{k87}  \mathbb{S}^n&\rightharpoonup \mathbb{S} &&\text{~weakly in~} L^{2(1-a)}(Q_{\delta,T}).
\end{align}
\end{subequations} 

Next, we show that
\begin{equation}\label{k8}
    |\mathbb{D}\bv| < M \qquad \textrm{ a.e. in } Q. 
\end{equation}
Indeed, using \eqref{est_5c} and \eqref{k88} together with Lebesgue dominated convergence theorem  we have that
\begin{align*}
    |\{ |\mathbb{D}\bv| \ge M\}| &= \int_Q \chi_{\{|\mathbb{D}\bv| \ge M\}} = \lim_{\omega \to 0_{+}} \int_Q \frac{(|\mathbb{D}\bv|-M)_{+}}{(|\mathbb{D}\bv|-M)_{+} + \omega}
    = \lim_{\omega \to 0_{+}} \lim_{n\to \infty} \int_Q \frac{(|\mathbb{D}\bv^n|-M-\tfrac{1}{n})_{+}}{(|\mathbb{D}\bv^n|-M-\tfrac{1}{n})_{+} + \omega} = 0,
\end{align*}
and \eqref{k8} follows.

Recalling that $\mathbb{S}^n=\mathbb{S}_n(\mathbb{D}\bv^n)$, see \eqref{galerkin3}, we conclude from \eqref{k88} and \eqref{k87} that
\begin{equation}
    \label{k89} \mathbb{S}^n=\mathbb{S}_n(\mathbb{D}\bv^n) \to \mathbb{S}(\mathbb{D}\bv)=\mathbb{S} \qquad \textrm{ a.e. in } Q.
\end{equation}
 By applying Fatou's lemma to \eqref{est_5}, we then observe that
\begin{equation}
    \label{k90} \int_Q |\mathbb{S}(\mathbb{D}\bv)|\le \frac{\|\bv_0\|^2_2}{m}\qquad \textrm{ i.e.,} \qquad  \mathbb{S} = \mathbb{S}(\mathbb{D}\bv)\in L^1(Q).
\end{equation}
Note that as $\mathbb{D}\bv = \mathbb{D}\bv^{\varepsilon}$, \eqref{k90} represents  a bound on $\{\mathbb{S}(\mathbb{D}\bv^{\varepsilon})\}$ that is uniform w.r.t.~$\varepsilon$.

Next, we multiply \eqref{galerkin1} by $\psi\in \mathcal{D}((0,T))$, integrate the result over (0,T) and study the limit as $n\to \infty$.  By applying \eqref{k85} to the first term, \eqref{k84} and \eqref{k86a} to the second term, Vitali's convergence  theorem and \eqref{k89} to the third term, and finally Vitali's convergence  theorem and \eqref{k88} to the fourth term, we obtain (using standard density arguments regarding the space of the test functions) that \eqref{wf-rovnice} holds. Also,  thanks to \eqref{k85}, $\bv \in C([0,T], L^2(\Omega))$ and the attainment of the initial condition follows from the construction.

The next two subsections focus on taking the limit $n\to \infty$ in suitable forms of uniform estimates established above, this time with the intention of  establishing estimates that are uniform w.r.t.~$\varepsilon$.

\subsection{Weak lower semicontinuity of \texorpdfstring{$\mathcal{Q}(\mathbb{D}\bv^n, \partial \mathbb{D}\bv^n)$}{Q}}\label{4.8} Here, we shall prove the following statement: if
\begin{align}
\label{dec:12a}
      \partial \mathbb{D}\bv^n &\rightharpoonup \partial \mathbb{D}\bv \qquad \textrm{weakly in } L^2\left(Q; \mathbb{R}^{d}\right) &\textrm{as } n\to \infty,\\
\label{dec:12b}
\mathbb{D}\bv^n &\to \mathbb{D}\bv \qquad \textrm{a.e. in } Q &\textrm{as } n\to\infty,
\end{align}
then
\begin{equation}
  \label{dec:11} \int_{Q_{\delta,T}} \mathcal{Q}(\mathbb{D}\bv, \partial \mathbb{D}\bv) \le \liminf_{n\to\infty} \int_{Q_{\delta,T}} \mathcal{Q}_n(\mathbb{D}\bv^n, \partial \mathbb{D}\bv^n).
\end{equation}
 Here, and in what follows, the symbol $\partial$ represents the spatial gradient $\nabla$ or the partial derivative  $\partial_t$ with respect to time. To prove  \eqref{dec:11}, we first observe from the definition of $\mathcal{Q}_n$, see \eqref{ox_22}, that, for $n\ge n_0$,
\begin{align} \label{k91}
    \int_{Q_{\delta,T}} \mathcal{Q}_n(\mathbb{D}\bv^n, \partial \mathbb{D}\bv^n) \ge     \int_{Q_{\delta,T}} \mathcal{Q}_{n_0}(\mathbb{D}\bv^n, \partial \mathbb{D}\bv^n).
\end{align}
The  expression on the right-hand side of this inequality can be identified with the square of the weighted $L^2$-norm of the form
$$
\oldnorm{\partial \mathbb{D}\bv^n}_{\mathcal{A}(\mathbb{D}\bv^n)}^2.
$$
 This means that
$$
\int_{Q_{\delta,T}} \mathcal{Q}_{n_0}(\mathbb{D}\bv^n, \partial \mathbb{D}\bv^n) = \oldnorm{\partial \mathbb{D}\bv^n}_{\mathcal{A}(\mathbb{D}\bv^n)}^2.
$$
By fixing $n_0$, the coefficients of  the fourth-order tensor $\mathcal{A}$, which we do not  state here  explicitly, are bounded. Then,  by arguing in the same manner as in \cite[Subsect. 4.7]{BuHrMa23}, we conclude from \eqref{dec:12a} and \eqref{dec:12b} that, for any $n_0\in \mathbb{N}$,
\begin{align*}
\int_{Q_{\delta,T}\cap\{|\mathbb{D}\bv|\le M-1/n_0\}} \mathcal{Q}(\mathbb{D}\bv, \partial \mathbb{D}\bv) \le
    \int_{Q_{\delta,T}} \mathcal{Q}_{n_0}(\mathbb{D}\bv, \partial \mathbb{D}\bv) &\le \liminf_{n\to\infty}  \int_{Q_{\delta,T}} \mathcal{Q}_{n_0}(\mathbb{D}\bv^n, \partial \mathbb{D}\bv^n) \\ &\overset{\eqref{k91}}{\le}
    \liminf_{n\to\infty} \int_{Q_{\delta,T}} \mathcal{Q}_n(\mathbb{D}\bv^n, \partial \mathbb{D}\bv^n).
\end{align*}
Since $n_0$ is arbitrary and \eqref{k8} holds, the monotone convergence theorem gives \eqref{dec:11}.

\subsection{Passage to the limit in the energy inequality}\label{4.9}

We first summarize the available estimates that are uniform with respect to $\varepsilon$. Then we take limit $n\to\infty$ in some previously established identities/inequalities but proceed differently in controlling the terms  that come from the convective term. This will lead to other sets of estimates that are uniform w.r.t.~$\varepsilon$.

We first recall the estimates established in Subsect.~\ref{4.7}. It follows from \eqref{k8} that
\begin{equation}
    \label{jul1} \|\mathbb{D}\bv^{\varepsilon}\|_{\infty}\le M.
\end{equation}
Setting $\mathbb{S}^{\varepsilon}:=\mathbb{S}(\mathbb{D}\bv^{\varepsilon})$, we also have  that
\begin{equation}
    \label{jul2} \int_Q |\mathbb{S}^{\varepsilon}|\le \frac{\|\bv_0\|^2_2}{m}.
\end{equation}

Second, referring to \eqref{est_2} and \eqref{est_3} (neglecting the first non-negative term) and letting $n\to \infty$, we obtain with help of \eqref{k84} and weak-lower semicontinuity of the appropriate norms that
\begin{align}\label{jul3}
    \sup_{t\in (0,T)} \Abs{\bv^\varepsilon(t, \cdot)}_{2}^2 &\le \Abs{\bv_0}_{2}^2, \\
    \varepsilon\int_0^T \|\mathbb{D}\bv^\varepsilon\|_4^4 &\le \Abs{\bv_0}_{2}^2. \label{jul4}
\end{align}

Third, we notice that we have, as a special case of \eqref{E2}, that
\begin{equation*}
\int_{Q} \mathcal{Q}_n(\mathbb{D}\bv^n, \nabla \mathbb{D}\bv^n) \le - \int_{Q} \partial_{x_s} v^n_k \partial_{x_k} v^n_i \partial_{x_s} v^n_i + \frac12 \Abs{\nabla\bv_0}_{2}^2.
\end{equation*}
Letting $n\to \infty$ and using \eqref{dec:11} with $\delta=0$ (noticing that the assumptions \eqref{dec:12a} and \eqref{dec:12b} are fulfilled by virtue of \eqref{k84} and \eqref{k88}) we get
\begin{equation}\label{jul5}
\int_{Q} \mathcal{Q}(\mathbb{D}\bv^\varepsilon, \nabla \mathbb{D}\bv^\varepsilon) \le - \int_{Q} \partial_{x_s} v^\varepsilon_k \partial_{x_k} v^\varepsilon_i \partial_{x_s} v^\varepsilon_k + \frac12 \Abs{\nabla\bv_0}_{2}^2 \overset{\eqref{jul1}}{\le} C(T,L,M) + \frac12 \Abs{\nabla\bv_0}_{2}^2.
\end{equation}

Fourth, neglecting the first three terms in \eqref{zvole_2} and recalling the definition of $I_n$ and the arguments following \eqref{zvole_2} we get
\begin{equation*}
    \int_0^T \|\partial_t \bv^n\|_2^2
        \le - \int_{Q} \bv^n \cdot (\mathbb{D}\bv^n)\partial_t \bv^n + C(m, L).
\end{equation*}
 By letting $n\to \infty$ in this inequality, using \eqref{k85}, \eqref{k86a} and \eqref{k86}, we obtain
\begin{equation}\label{jul6}\begin{split}
    \int_0^T \|\partial_t \bv^\varepsilon\|_2^2
        &\le - \int_{Q} \bv^\varepsilon \cdot (\mathbb{D}\bv^\varepsilon)\partial_t \bv^\varepsilon + C(m, L) \overset{\eqref{jul1}}{\le} M\int_0^T\|\bv^\varepsilon\|_2 \|\partial_t \bv^\varepsilon\|_2 + C(m,L) \\ &\le \frac{1}{2} \int_0^T \|\partial_t \bv^\varepsilon\|_2^2 + \frac{M^2}{2} \int_0^T \|\bv^\varepsilon\|_2^2 + C(m,L).
    \end{split}
\end{equation}
This, together with \eqref{jul3}, implies that
\begin{equation}
    \label{jul7}
    \int_0^T \|\partial_t \bv^\varepsilon\|_2^2 \le C(m,L, T,\|\bv_0\|_2^2).
\end{equation}

Fifth, using \eqref{k2dif} where we neglect the first and the third terms that are nonnegative (the non-negativity of the third term follows from the fact that the corresponding operator is monotone), we arrive at
\begin{equation}\label{jul8}
       \int_{Q} \zeta (\eth_{\tau}{\mathbb{S}}_n(\mathbb{D}\bv^n), \eth_{\tau} \mathbb{D}\bv^n) 
        \le \int_{Q} \zeta' |\eth_{\tau} \bv^n|^2 - \int_{Q} \zeta (\eth_{\tau} \bv^n \otimes\eth_{\tau} \bv^n)\cdot (\mathbb{D}\bv^n).
\end{equation}
Recalling the definition of $\eth_{\tau}\vect{u}$, see \eqref{kdef}, we notice that thanks to the  monotonicity property of the mapping $\mathbb{S}$ resp. $\mathbb{S}_n$, see Lemma~\ref{monotone},
$$
\eth_{\tau}{\mathbb{S}}_n(\mathbb{D}\bv^n) \cdot \eth_{\tau} \mathbb{D}\bv^n = \left({\mathbb{S}}_n(\mathbb{D}\bv^n(t+\tau, \cdot)) - {\mathbb{S}}_n(\mathbb{D}\bv^n(t, \cdot))\right)\cdot \left(\mathbb{D}\bv^n(t+\tau, \cdot) - \mathbb{D}\bv^n(t, \cdot)\right) \ge 0,
$$
i.e., the term on the left-hand side of \eqref{jul8} is nonnegative.
Keeping $\tau$ fixed, we take limit $n\to\infty$ in \eqref{jul8}.
Owing to \eqref{k88} and \eqref{k89}, we can apply Fatou's lemma on the left hand side, while on the right-hand side we use the strong convergences \eqref{k86a} and \eqref{k86}. Doing so, we obtain
\begin{equation}\label{jul9}
       \int_{Q} \zeta (\eth_{\tau}{\mathbb{S}}(\mathbb{D}\bv), \eth_{\tau} \mathbb{D}\bv)
        \le \int_{Q} \zeta' |\eth_{\tau} \bv|^2 - \int_{Q} \zeta (\eth_{\tau} \bv \otimes\eth_{\tau} \bv)\cdot (\mathbb{D}\bv),
\end{equation}
where $\bv=\bv^{\varepsilon}$.

Next, we divide~\eqref{jul9} by $\tau^2$ and let $\tau\to 0$. By virtue of \eqref{jul7} and \eqref{jul1} the right-hand side of \eqref{jul9} is bounded by a constant $C(m,M,L,T, \|\bv_0\|_2^2, \delta^{-1})$. The nonnegative term on the left-hand side is thus uniformly bounded with respect to  the time difference. Consequently,
\begin{equation}\label{jul10}
       \int_{Q} \zeta (\partial_t {\mathbb{S}}(\mathbb{D}\bv^\varepsilon), \partial_t \mathbb{D}\bv^{\varepsilon})
        \le C(m,M,L,T, \|\bv_0\|_2^2, \delta^{-1}).
\end{equation}
Finally, using \eqref{ox_18} and recalling the properties of the cut-off function $\zeta$, we conclude from \eqref{jul10} that
\begin{equation}\label{jul11}
       \int_{Q_{\delta,T}} \mathcal{Q}(\mathbb{D}\bv^\varepsilon, \partial_t\mathbb{D}\bv^\varepsilon)
        \le C(m,M,L,T, \|\bv_0\|_2^2, \delta^{-1}).
\end{equation}

The sixth,  and final, estimate concerns $\mathbb{S}^{\varepsilon}:=\mathbb{S}(\mathbb{D}\bv^{\varepsilon})$. It follows from \eqref{jul5} and \eqref{jul11} that
\begin{equation}\label{jul12}
\int_{Q_{\delta,T}} \mathcal{Q}(\mathbb{D}\bv^\varepsilon, \nabla \mathbb{D}\bv^\varepsilon) + \mathcal{Q}(\mathbb{D}\bv^\varepsilon, \partial_t \mathbb{D}\bv^\varepsilon) \le C(m,M,L,T, \|\bv_0\|_2^2, \delta^{-1}).
\end{equation}
 This means that we are  in the same position as in Subsect.~\ref{4.6}, see the estimate \eqref{k72}, but this time the estimate is independent of $\varepsilon$. Following the steps performed in Subsect.~\ref{4.6} without any additional change, we conclude that
\begin{equation}\label{jul13}
\textrm{if } \quad a\in (0,1/2) \quad \textrm{ then } \quad \|\mathbb{S}^\varepsilon\|_{L^{2(1-a)}(Q_{\delta,T})} \le C(m,M,L,T, \|\bv_0\|_2^2, \delta^{-1}).
\end{equation}

\subsection{The limit \texorpdfstring{$\varepsilon\to 0_{+}$}{e}}\label{4.10}

It follows from the estimates \eqref{jul1}, \eqref{jul7}, the Aubin--Lions compactness lemma, and further from the estimates \eqref{jul13} that there is a subsequence $\left\{(\bv^\ell, \mathbb{S}^\ell)\right\}_{\ell=1}^{\infty}:=\left\{(\bv^{\varepsilon_\ell}, \mathbb{S}^{\varepsilon_\ell})\right\}_{\ell=1}^{\infty}$ of $\left\{(\bv^\varepsilon, \mathbb{S}^\varepsilon)\right\}_{\varepsilon>0}$ such that as $\ell \to \infty$ ($\iff \varepsilon_{\ell} \to 0_{+}$)
\begin{subequations}\label{lim_july}
\begin{alignat}{2}
    \label{jul14}
    \bv^\ell&\rightharpoonup \bv &&\quad \text{weakly$^*$ in~} L^{\infty}\big(0,T, W^{1,\infty}_{per,div}\big) \quad \textrm{ and } \quad \|\mathbb{D}\bv\|_{\infty,Q}\le M, \\ \label{jul15}
    \partial_t{\bv}^\ell &\rightharpoonup \, \partial_t {\bv} &&\quad \text{weakly in~} L^{2}(Q) \qquad \implies \qquad \bv\in C([0,T], L^2_{per,div})\\  \label{jul16}
    \bv^n&\to \bv &&\quad \text{strongly in~} L^q(Q), \qquad (q\in [1, \infty)),
    \\\label{jul17}
\mathbb{S}^\ell&\rightharpoonup \overline{\mathbb{S}} &&\quad \text{weakly in~} L^{2(1-a)}(Q_{\delta,T}) \qquad \textrm{for any } \delta>0 \textrm{ fixed},
\end{alignat}
\end{subequations}
where $\overline{\mathbb{S}}$ is an element of $L^{2(1-a)}(Q_{\delta,T})$, but we have to show that $\overline{\mathbb{S}}=\mathbb{S}(\mathbb{D}\bv)$.

Next, multiplying \eqref{wf-rovnice} by $\psi\in L^{\infty}(0,T)$, integrating the result over $(\delta,T)$, we study the limit as $\ell\to \infty$. Applying \eqref{jul15} to the first term, \eqref{jul14} and \eqref{jul16} to the second term, \eqref{jul17} to the third term, and using the estimate \eqref{jul4} to conclude that the fourth term vanishes as $\varepsilon \to 0_{+}$, we obtain
\begin{equation}
    \begin{split}
       \int_\delta^T (\partial_t \bv,\vect{\varphi}) - (\bv\otimes\bv,\mathbb{D}\vect{\varphi}) &+ (\overline{\mathbb{S}}, \mathbb{D}\vect{\varphi}) = 0 \\ &\textrm{ for all } \vect{\varphi}\in L^{2}(0,T; W^{1,2}_{per,div}) \textrm{ with } \mathbb{D}\vect{\varphi} \in L^{\infty}(Q).
    \end{split}\label{wf-jul}
\end{equation}

Also,  thanks to \eqref{jul15}, not only   do we have that $\bv \in C([0,T], L^2(\Omega))$ but we can also take the limit $\ell \to \infty$ in (the weak form of)
$$
\bv^{\ell}(t, \cdot) - \bv_{0,\varepsilon_\ell} = \int_0^t \partial_t\bv^{\ell}
$$
and conclude that \eqref{initial} holds.

Next, we apply a variant of the monotone operator method to conclude that $\overline{\mathbb{S}} = \mathbb{S}(\mathbb{D}\bv)$ a.e. in $Q$. To do so, we set $\vect{\varphi} = \bv^{\ell}(t, \cdot)$ in \eqref{wf-rovnice} and integrate the result over $(\delta,T)$. This leads to
\begin{equation}\label{jul18}
    \frac12 \|\bv^\ell(T, \cdot)\|_2^2 + \int_{Q_{\delta,T}} \mathbb{S}(\mathbb{D} \bv^\ell) \cdot \mathbb{D}\bv^\ell + \varepsilon_{\ell} \int_{Q_{\delta,T}} (1+|\mathbb{D}\bv^\ell|^2) |\mathbb{D}\bv^\ell|^2 = \frac12 \|\bv^\ell(\delta, \cdot)\|_2^2.
\end{equation}
We also set $\vect{\varphi} = \bv$ in \eqref{wf-jul} and arrive at
\begin{equation}\label{jul19}
    \frac12 \|\bv(T, \cdot)\|_2^2 + \int_{Q_{\delta,T}} \overline{\mathbb{S}} \cdot \mathbb{D}\bv= \frac12 \|\bv(\delta, \cdot)\|_2^2.
\end{equation}
Since $\mathbb{S}(\mathbb{D}\bv)$ generates a monotone operator, see Lemma~\ref{monotone}, we observe, using also \eqref{jul18}, in which we neglect the third (nonnegative) term, that
\begin{equation} \label{jul19a}
\begin{split}
   0 &\le \int_{Q_{\delta,T}} (\mathbb{S}(\mathbb{D}\bv^{\ell}) - \mathbb{S}(\mathbb{B})\cdot (\mathbb{D}\bv^{\ell} - \mathbb{B}) \\
   &\le \frac12 \|\bv^\ell(\delta, \cdot)\|_2^2 - \frac12 \|\bv^\ell(T, \cdot)\|_2^2 - \int_{Q_{\delta,T}} \mathbb{S}(\mathbb{D}\bv^{\ell})\cdot \mathbb{B} -
   \int_{Q_{\delta,T}} \mathbb{S}(\mathbb{B})\cdot (\mathbb{D}\bv^\ell - \mathbb{B}),
\end{split}
\end{equation}
which holds for any $\mathbb{B}\in L^{\infty} (Q)$ with $\|\mathbb{B}\|_{\infty, Q}<M$ satisfying $\mathbb{S}(\mathbb{B})\in L^1(Q_{\delta,T})$.

Letting $\ell\to \infty$ in \eqref{jul19a}, using \eqref{jul14}, \eqref{jul16}, \eqref{jul17} and the facts that $\bv^{\ell}, \bv\in C([0,T]; L^2_{div,per})$, we obtain
\begin{equation} \label{jul20}
    \begin{split}
   0
   &\le \frac12 \|\bv(\delta, \cdot)\|_2^2 - \frac12 \|\bv(T, \cdot)\|_2^2 - \int_{Q_{\delta,T}} \overline{\mathbb{S}} \cdot \mathbb{B} -
   \int_{Q_{\delta,T}} \mathbb{S}(\mathbb{B})\cdot (\mathbb{D}\bv - \mathbb{B}) \\
   &\overset{\eqref{jul19}}{=} \int_{Q_{\delta,T}} \overline{\mathbb{S}} \cdot \mathbb{D}\bv - \int_{Q_{\delta,T}} \overline{\mathbb{S}} \cdot \mathbb{B} -
   \int_{Q_{\delta,T}} \mathbb{S}(\mathbb{B})\cdot (\mathbb{D}\bv - \mathbb{B}),
   \end{split}
\end{equation}
which can be rewritten as
\begin{equation}\label{E71}
   \int_{Q_{\delta,T}} (\overline{\mathbb{S}} - \mathbb{S}(\mathbb{B})\cdot (\mathbb{D}\bv - \mathbb{B}) \ge 0 \quad \textrm{ for all $\mathbb{B}$ specified above.}
\end{equation}
Next, take any $\kappa_0<M/2$ and consider $\mathbb{A}\in L^{\infty}(Q)$  such that $\|\mathbb{A}\|_{\infty}\le 1$. Setting
$$
\mathbb{B}:= \D\bv \chi_{\{|\D\bv|>M-2\kappa\}} + \left(\D\bv - \kappa\mathbb{A}\right)\chi_{\{|\D\bv|\le M-2\kappa\}} \qquad \textrm{ with } \, \kappa\in (0,\kappa_0),
$$
we see that $\mathbb{S}(\mathbb{B})\in L^1(Q_{\delta,T})$. Upon inserting such a $\mathbb{B}$ into \eqref{E71}, we get
\begin{equation*}
   \kappa \int_{Q_{\delta,T}\cap\{|\D\bv|\le M-2\kappa_0\}} (\overline{\mathbb{S}} - \mathbb{S}(\D\bv - \kappa\mathbb{A}))\cdot \mathbb{A} \ge 0.
\end{equation*}
Dividing  this by $\kappa$ and  then letting $\kappa\to 0_{+}$, we deduce  that
\begin{equation*}
   \int_{Q_{\delta,T}\cap\{|\D\bv|\le M-2\kappa_0\}} (\overline{\mathbb{S}} - \mathbb{S}(\D\bv))\cdot \mathbb{A} \ge 0
\qquad \textrm{ for any } \mathbb{A} \textrm{  such that } \|\mathbb{A}\|_{\infty}\le 1.
\end{equation*}
Hence, setting
$$
\mathbb{A}:=-\frac{\overline{\mathbb{S}} - \mathbb{S}(\D\bv)}{1+|\overline{\mathbb{S}} - \mathbb{S}(\D\bv)|},
$$
we get
$$
\overline{\mathbb{S}}= \mathbb{S}(\mathbb{D}\bv) \quad \textrm{ a.e.  in } Q_{\delta,T}\cap \{\|\D\bv|\le M-2\kappa_0\}.
$$
Since $|\D\bv|< M$ a.e. in Q and $\kappa_0\in (0, M/2)$ is arbitrary, we finally come to the conclusion that
\begin{equation}
    \label{jul21}
    \overline{\mathbb{S}}= \mathbb{S}(\mathbb{D}\bv) \quad \textrm{ a.e.  in } Q_{\delta,T}.
\end{equation}

It remains to extend the validity of \eqref{wf-jul} (with the relation \eqref{jul21}  we have just proved) on the whole time interval $(0,T)$.  To this end, it suffices to show that $\mathbb{S}(\mathbb{D}\bv)\in L^1(Q)$. We  know however from \eqref{jul2} that
\begin{equation}
    \label{k90new} \int_{Q_{\delta,T}} |\mathbb{S}(\mathbb{D}\bv^\ell)| \le \int_Q |\mathbb{S}(\mathbb{D}\bv^\ell)|\le \frac{\|\bv_0\|^2_2}{m}.
\end{equation}
 Thanks to \eqref{jul17} and \eqref{jul21}, $\mathbb{S}(\mathbb{D}\bv^\ell) \rightharpoonup \mathbb{S}(\mathbb{D})$ weakly in $L^1(Q_{\delta,T})$. Using the convexity of the $L^1$-norm, we conclude from \eqref{k90new} that
\begin{equation}
    \label{jul22} \int_{Q_{\delta,T}} |\mathbb{S}(\mathbb{D}\bv)| \le \liminf_{\ell\to\infty} \int_{Q_{\delta,T}} |\mathbb{S}(\mathbb{D}\bv^\ell)|\le \frac{\|\bv_0\|^2_2}{m}.
\end{equation}
Since \eqref{jul22} holds for any $\delta>0$ and the upper bound in \eqref{jul22} is independent of $\delta$, we  deduce that, indeed, $\mathbb{S}(\mathbb{D}\bv)\in L^1(Q)$.

The proof of Theorem~\ref{main_thm} is complete.

\section{Concluding remarks}
\label{sec:5}

We have proved that  the spatially periodic initial-boundary-value problem for the activated Euler fluid model~\eqref{zv_17} is well-posed in the sense of Hadamard. This, in particular, means that \emph{there exists a unique weak solution satisfying the energy equality, globally-in-time}. As the activation parameter $m$ in~\eqref{zv_17} can be arbitrarily large, the activated Euler fluid model may be seen to be a suitable (well-posed) alternative to the classical Euler fluid model.  A number of open problems concerning finite-time singularity formation for an Euler fluid formulated in   the interesting recent survey~\cite{DrEl23} could be attacked from this alternative perspective, see also for example~\cite{Co19}.

Our   forthcoming objective   will be to consider internal flows in bounded containers and extend the result   established here to   various initial-boundary-value problems   involving different boundary conditions. It seems   to be an advantageous feature of the activated Euler fluid   model that   it allows one   to incorporate, in a very simple manner, standard boundary conditions, such as no-slip, Navier's slip, or perfect slip.





\section{Appendix}\label{Appendix}

\noindent In this part of the text we study the properties of the functions ${\mathbb{S}}_n$ and ${\mathbb{S}}$, considered as mappings from $\R^{3\times 3}_{sym}$ to $\R^{3\times 3}_{sym}$, which characterize the response of an activated Euler fluid and its $n$-approximation used in the analysis; see \eqref{constit_old} and \eqref{galerkin3} for their   definitions. For $n\in \mathbb{N}$, $0<m<M<\infty$ and $a>0$, these functions are defined as follows:
\begin{equation}
{\mathbb{S}}(\mathbb{D}) := \frac{(|\mathbb{D}|-m)_{+}}{\left( M^a - |\mathbb{D}|^{a}\right)^{\frac{1}{a}}}\frac{\mathbb{D}}{|\mathbb{D}|} \qquad  \textrm{ for } \mathbb{D}\in \mathcal{B}_M:=\{\mathbb{B}\in \mathbb{R}^{3\times3}_{sym}; |\mathbb{B}|<M\}, \label{ox_11}
\end{equation}
and
\begin{equation}
{\mathbb{S}}_n(\mathbb{D}) := \frac{(|\mathbb{D}|-m)_{+}}{\left( M^a - (\min\{|\mathbb{D}|,M-1/n\})^{a}\right)^{\frac{1}{a}}}\frac{\mathbb{D}}{|\mathbb{D}|} \qquad\qquad  \textrm{ for } \mathbb{D}\in \mathbb{R}^{3\times3}_{sym}.\label{ox_12}
\end{equation}

 By setting
\begin{equation}
    c_{n,M,a}:=M^a - (M-1/n)^a, \label{ox_13}
\end{equation}
we  have that
\begin{align}
    {\mathbb{S}}_n(\mathbb{D}) &= {\mathbb{S}}(\mathbb{D}) &&\textrm{on the set } \{ |\mathbb{D}|\le M-1/n \}, \label{ox_14} \\
    {\mathbb{S}}_n(\mathbb{D}) &=\frac{(|\mathbb{D}|-m)_{+}}{[c_{n,M,a}]^{\frac{1}{a}}} \frac{\mathbb{D}}{|\mathbb{D}|} &&\textrm{on the set } \{ |\mathbb{D}|> M-1/n \}.\label{ox_15}
\end{align}
We are interested in   differentiability properties of these functions. In what follows, the symbol $\partial$ represents the spatial gradient $\nabla$ or the partial derivative $\partial_t$ with respect to time.

For $|\mathbb{D}|>m$ (and $|\mathbb{D}|<M$), we have
\begin{equation}\notag
   \partial{\mathbb{S}}(\mathbb{D}) = \frac{(|\mathbb{D}|-m)}{\left( M^a - |\mathbb{D}|^{a}\right)^{\frac{1}{a}}}\frac{\partial\mathbb{D}}{|\mathbb{D}|} + \left( \frac{(|\mathbb{D}|-m)|\mathbb{D}|^{a-2}}{\left( M^a - |\mathbb{D}|^{a}\right)^{\frac{1}{a}+1}|\mathbb{D}|} - \frac{(|\mathbb{D}|-m)}{\left( M^a - |\mathbb{D}|^{a}\right)^{\frac{1}{a}}|\mathbb{D}|^3} + \frac{1}{\left( M^a - |\mathbb{D}|^{a}\right)^{\frac{1}{a}}|\mathbb{D}|^2}\right)(\mathbb{D}\cdot\partial \mathbb{D}) \mathbb{D},
\end{equation}
which simplifies to
\begin{equation}\label{ox_17}
   \partial{\mathbb{S}}(\mathbb{D}) = \frac{(|\mathbb{D}|-m)}{\left( M^a - |\mathbb{D}|^{a}\right)^{\frac{1}{a}}}\frac{\partial\mathbb{D}}{|\mathbb{D}|} + \frac{(|\mathbb{D}|-m)|\mathbb{D}|^{a} + m(M^a - |\mathbb{D}|^a)}{\left( M^a - |\mathbb{D}|^{a}\right)^{\frac{1}{a}+1}|\mathbb{D}|^3} (\mathbb{D}\cdot\partial \mathbb{D})\mathbb{D}.
\end{equation}
This implies  that
\begin{equation}\label{ox_18}
   \partial{\mathbb{S}}(\mathbb{D})\cdot\partial\mathbb{D}  = \frac{(|\mathbb{D}|-m)}{\left( M^a - |\mathbb{D}|^{a}\right)^{\frac{1}{a}}}\frac{|\partial\mathbb{D}|^2}{|\mathbb{D}|} + \frac{(|\mathbb{D}|-m)|\mathbb{D}|^{a} + m(M^a - |\mathbb{D}|^a)}{\left( M^a - |\mathbb{D}|^{a}\right)^{\frac{1}{a}+1}|\mathbb{D}|^3} (\mathbb{D}\cdot\partial \mathbb{D})^2=:\mathcal{Q}(\mathbb{D},\partial\mathbb{D})\ge 0.
\end{equation}
Using again \eqref{ox_17}, we also observe that
\begin{equation}\label{ox_19}
\begin{split}
   |\partial{\mathbb{S}}(\mathbb{D})|^2 &=\partial{\mathbb{S}}(\mathbb{D})\cdot\partial{\mathbb{S}}(\mathbb{D}) = \frac{(|\mathbb{D}|-m)^2}{\left( M^a - |\mathbb{D}|^{a}\right)^{\frac{2}{a}}}\frac{|\partial\mathbb{D}|^2}{|\mathbb{D}|^2} \\
   & + \left(\frac{2(|\mathbb{D}|-m)\left[(|\mathbb{D}|-m)|\mathbb{D}|^{a} + m(M^a - |\mathbb{D}|^a)\right]}{\left( M^a - |\mathbb{D}|^{a}\right)^{\frac{2}{a}+1}|\mathbb{D}|^4} + \frac{\left[(|\mathbb{D}|-m)|\mathbb{D}|^{a} + m(M^a - |\mathbb{D}|^a)\right]^2}{\left( M^a - |\mathbb{D}|^{a}\right)^{\frac{2}{a}+2}|\mathbb{D}|^4}\right) (\mathbb{D}\cdot\partial \mathbb{D})^2.
\end{split}
\end{equation}
 By setting
\begin{align}
    \alpha(|\mathbb{D}|, m, M)&:= |\mathbb{D}|^{-1} (|\mathbb{D}|-m)(M^a - |\mathbb{D}|^a), \label{zv_22}\\
    \beta(|\mathbb{D}|, m, M)&:= |\mathbb{D}|^{-1} \left(2(|\mathbb{D}|-m)(M^a - |\mathbb{D}|^a) + (|\mathbb{D}|-m)|\mathbb{D}|^a + m(M^a-|\mathbb{D}|^a)\right), \label{zv_23}
\end{align}
we can rewrite \eqref{ox_19}, after multiplying it by $(M^a - |\mathbb{D}|^a)^{1+\frac{1}{a}}$, as
{\small
\begin{equation}\label{ox_19a}
   (M^a - |\mathbb{D}|^a)^{1+\frac{1}{a}}|\partial{\mathbb{S}}(\mathbb{D})|^2 =\alpha(|\mathbb{D}|, m, M)\frac{(|\mathbb{D}|-m)|\partial\mathbb{D}|^2}{(M^a - |\mathbb{D}|^a)^{\frac{1}{a}}} + \beta(|\mathbb{D}|, m, M)\frac{(|\mathbb{D}|-m)|\mathbb{D}|^{a} + m(M^a - |\mathbb{D}|^a)}{\left( M^a - |\mathbb{D}|^{a}\right)^{\frac{1}{a}+1}|\mathbb{D}|^3} (\mathbb{D}\cdot\partial \mathbb{D})^2.
\end{equation}}

\smallskip
\noindent
 By recalling the definition of $\mathcal{Q}(\mathbb{D},\partial\mathbb{D})$ introduced in \eqref{ox_18},
it follows from \eqref{ox_19a} that
\begin{equation}\label{ox_19b}
   (M^a - |\mathbb{D}|^a)^{1+\frac{1}{a}}|\partial{\mathbb{S}}(\mathbb{D})|^2 \le \max\{\alpha(|\mathbb{D}|, m, M), \beta(|\mathbb{D}|, m, M)\} \, \mathcal{Q}(\mathbb{D},\partial\mathbb{D}).
\end{equation}

We proceed in a similar manner for the function ${\mathbb{S}}_n$. For $|\mathbb{D}|>m$ and $|\mathbb{D}|<M-1/n$, \eqref{ox_14} holds and hence the above calculations remain unchanged. For $|\mathbb{D}|> M-1/n$, we use the formula \eqref{ox_15} and obtain
\begin{equation}\label{ox_20}
    \begin{split}
   \partial{\mathbb{S}}_n(\mathbb{D}) &= \frac{1}{[c_{n,M,a}]^{\frac{1}{a}}}\left( \frac{(|\mathbb{D}|-m)\partial\mathbb{D}}{|\mathbb{D}|} + \left( \frac{1}{|\mathbb{D}|^{2}} - \frac{(|\mathbb{D}|-m)}{|\mathbb{D}|^3}\right)(\mathbb{D}\cdot\partial \mathbb{D}) \mathbb{D}\right)\\
   &= \frac{1}{[c_{n,M,a}]^{\frac{1}{a}}}\left( \frac{(|\mathbb{D}|-m)\partial\mathbb{D}}{|\mathbb{D}|} + \frac{m(\mathbb{D}\cdot\partial \mathbb{D}) \mathbb{D}}{|\mathbb{D}|^3}\right).
    \end{split}
\end{equation}
Then,
\begin{equation}\label{ox_21}
   \partial{\mathbb{S}}_n(\mathbb{D})\cdot\partial\mathbb{D}  = \frac{1}{[c_{n,M,a}]^{\frac{1}{a}}}\left(\frac{(|\mathbb{D}|-m)|\partial\mathbb{D}|^2}{|\mathbb{D}|} + \frac{m(\mathbb{D}\cdot\partial \mathbb{D})^2}{|\mathbb{D}|^3}\right). 
\end{equation}
At this point, recalling definition of $\mathcal{Q}(\mathbb{D},\partial\mathbb{D})$ (as introduced in \eqref{ox_18}), we set
\begin{equation} \label{ox_22}
\mathcal{Q}_n(\mathbb{D},\partial\mathbb{D}):=
    \begin{cases}
        &\mathcal{Q}(\mathbb{D},\partial\mathbb{D}) \qquad \qquad \qquad\qquad \qquad\qquad \,\,\quad\textrm{ if  } |\mathbb{D}|<M-1/n, \\
        & \frac{1}{[c_{n,M,a}]^{\frac{1}{a}}}\left(\frac{(|\mathbb{D}|-m)|\partial\mathbb{D}|^2}{|\mathbb{D}|} + \frac{m(\mathbb{D}\cdot\partial \mathbb{D})^2}{|\mathbb{D}|^3}\right) \qquad \, \textrm{ if  } |\mathbb{D}|\ge M-1/n.
    \end{cases}
\end{equation}
Hence,
\begin{equation}
    \label{ox_22a} \partial \mathbb{S}_n(\mathbb{D}) \cdot \partial \mathbb{D} = \mathcal{Q}_n(\mathbb{D},\partial\mathbb{D}).
\end{equation}
Using again \eqref{ox_20}, we also observe that, for $|\mathbb{D}|\ge M- 1/n$,
\begin{equation}\label{ox_23}
    \begin{split}
   |\partial{\mathbb{S}}_n(\mathbb{D})|^2 &=\partial{\mathbb{S}}_n(\mathbb{D})\cdot\partial{\mathbb{S}}_n(\mathbb{D})= \frac{1}{[c_{n,M,a}]^{\frac{2}{a}}} \left( \frac{(|\mathbb{D}|-m)^2}{|\mathbb{D}|^{2}} |\partial\mathbb{D}|^2 + \frac{m(2|\mathbb{D}|-m)}{|\mathbb{D}|^4} (\mathbb{D}\cdot\partial \mathbb{D})^2\right) \\
   &\le \frac{1}{[c_{n,M,a}]^{\frac{1}{a}}}\frac{2|\mathbb{D}|-m}{|\mathbb{D}|} \mathcal{Q}_n(\mathbb{D},\partial\mathbb{D}).
   \end{split}
\end{equation}

Our next aim is to show that
\begin{equation}
    \frac{|\partial \mathbb{S}_n(\mathbb{D})|^2}{(1+|\mathbb{S}_n(\mathbb{D})|)^{1+a}} \le C \mathcal{Q}_n(\mathbb{D},\partial\mathbb{D}) \quad \textrm{ for all } |\mathbb{D}|\ge \tfrac{M+m}{2}. \label{ox_24}
\end{equation}
 Thanks to the definition of $\mathbb{S}_n$, we consider two subcases: (1) $\tfrac{M+m}{2}\le |\mathbb{D}|< M- \tfrac{1}{n}$, and (2) $|\mathbb{D}|\ge M-\tfrac{1}{n}$.

First, if $\boxed{\tfrac{M+m}{2}\le |\mathbb{D}|< M- \tfrac{1}{n}}$, then \eqref{ox_14} and \eqref{ox_11} imply that
\begin{equation}
    \label{ox_25} |\mathbb{S}_n(\mathbb{D})|^a = \frac{(|\mathbb{D}|-m)^a}{(M^a - |\mathbb{D}|^a)} \quad\implies\quad 1+ |\mathbb{S}_n(\mathbb{D})|^a = \frac{(M^a - |\mathbb{D}|^a)+(|\mathbb{D}|-m)^a}{(M^a - |\mathbb{D}|^a)}.
\end{equation}
As $|\mathbb{D}|\ge \tfrac{M+m}{2}$, we observe (using $M^a - |\mathbb{D}|^a>0$ and $(|\mathbb{D}|-m)\ge \frac{M-m}{2}$ to get the first inequality, and $|\mathbb{D}|^a - (|\mathbb{D}|-m)^a >0$ to get the second inequality) that
\begin{equation}
    \label{ox_26} \frac{\left(\frac{M-m}{2}\right)^a}{M^a - |\mathbb{D}|^a} \le 1+|\mathbb{S}_n(\mathbb{D})|^a \le \frac{M^a}{M^a - |\mathbb{D}|^a}.
\end{equation}
Hence,
\begin{equation}
    \label{ox_27} \frac{\left(\frac{M-m}{2}\right)^{1+a}}{(1+|\mathbb{S}_n(\mathbb{D})|^a)^{1+\frac{1}{a}}} \le \left(M^a - |\mathbb{D}|^a\right)^{1+\frac{1}{a}} \le \frac{M^{1+a}}{(1+|\mathbb{S}_n(\mathbb{D})|^a)^{1+\frac{1}{a}}}.
\end{equation}
With this piece of information, referring to \eqref{ox_19b} and noticing that on the considered range of $|\mathbb{D}|$ the functions $\alpha$ and $\beta$ are bounded by some constant $C=C(m,M,a)$, we conclude that
\begin{equation}
    \left(\frac{M-m}{2}\right)^{1+a}\frac{|\partial \mathbb{S}_n(\mathbb{D})|^2}{(1+|\mathbb{S}_n(\mathbb{D})|^a)^{1+\frac{1}{a}}} \le C(m,M,a) \mathcal{Q}_n(\mathbb{D},\partial\mathbb{D}),\label{ox_28}
\end{equation}
which implies \eqref{ox_24} on the  interval considered.

Second, if $\boxed{|\mathbb{D}|\ge M- \tfrac{1}{n}}$, then we observe that \eqref{ox_15} implies that
\begin{equation}
    \label{ox_29} |\mathbb{S}_n(\mathbb{D})| = \frac{(|\mathbb{D}|-m)}{[c_{n,M,a}]^{\frac{1}{a}}} \quad\implies\quad (1+ |\mathbb{S}_n(\mathbb{D})|)^{a+1} = \frac{([c_{n,M,a}]^{\frac{1}{a}} + |\mathbb{D}| - m )^{a+1}}{[c_{n,M,a}]^{1+\frac{1}{a}}}.
\end{equation}
Hence,
\begin{equation}
    \frac{|\partial \mathbb{S}_n(\mathbb{D})|^2}{(1+|\mathbb{S}_n(\mathbb{D})|)^{1+a}} \le \frac{2[c_{n,M,a}]}{\left([c_{n,M,a}]^{\frac{1}{a}} + |\mathbb{D}|-m\right)^{a+1}} \mathcal{Q}_n(\mathbb{D},\partial\mathbb{D}) \le \frac{8[c_{n,M,a}]}{(M-m)^{a+1}} \mathcal{Q}_n(\mathbb{D},\partial\mathbb{D}),\label{ox_30}
\end{equation}
where we used the inequality $|\mathbb{D}|-m + [c_{n,M,a}]^{\frac{1}{a}} \ge M-\tfrac{1}{n}-m \ge (M-m)/2$, valid on the considered range of $|\mathbb{D}|\ge M-1/n$. Note that  thanks to the definition of $c_{n,M,a}$, the coefficient in front of $\mathcal{Q}_n(\mathbb{D},\partial\mathbb{D})$ in \eqref{ox_30} can be made arbitrarily small. In any case, the proof of \eqref{ox_24} is complete.

\begin{lemma}\label{monotone}
Let $a>0$ and $0<m<M<\infty$. Then, the mappings $\mathbb{S}:\mathcal{B}_M \to\R^{3\times 3}_{sym}$ and $\mathbb{S}_n:\R^{3\times 3}_{sym} \to\R^{3\times 3}_{sym}$ defined in \eqref{ox_11} and \eqref{ox_12}  satisfy
\begin{align}
\bigl(\mathbb{S}(\mathbb{D}_1)-\mathbb{S}(\mathbb{D}_2)\bigr)\cdot(\mathbb{D}_1-\mathbb{D}_2) &\ge 0 \qquad \textrm{for all } \mathbb{D}_1, \mathbb{D}_2\in\mathcal{B}_M, \label{dec:5} \\
\bigl(\mathbb{S}_n(\mathbb{D}_1)-\mathbb{S}_n(\mathbb{D}_2)\bigr)\cdot(\mathbb{D}_1-\mathbb{D}_2) &\ge 0 \qquad \textrm{for all } \mathbb{D}_1, \mathbb{D}_2\in\R^{3\times 3}_{sym}. \label{dec:5M}
\end{align}
\end{lemma}
\begin{proof} We start by observing that the functions $\mathbb{S}$ and $\mathbb{S}_n$ have the following structure:
$$
    \mathbb{S}(\mathbb{D})=f(|\mathbb{D}|)\frac{\mathbb{D}}{|\mathbb{D}|} \qquad \textrm{ and } \qquad \mathbb{S}_n(\mathbb{D})=f_n(|\mathbb{D}|)\frac{\mathbb{D}}{|\mathbb{D}|},
$$
whereas the functions $f$ 
and $f_n$ 
 are nondecreasing on their domains. Therefore, using the Cauchy--Schwarz inequality, we have
%
%
%
%
\begin{align*}
    (\mathbb{S}(\mathbb{D}_1)-\mathbb{S}(\mathbb{D}_2)) \cdot (\mathbb{D}_1-\mathbb{D}_2) &= f(|\mathbb{D}_1|)\frac{\mathbb{D}_1}{|\mathbb{D}|_1} \cdot \mathbb{D}_1 -f(|\mathbb{D}_1|)\frac{\mathbb{D}_1}{|\mathbb{D}_1|}\cdot \mathbb{D}_2-f(|\mathbb{D}_2|)\frac{\mathbb{D}_2}{|\mathbb{D}_2|}\cdot \mathbb{D}_1+f(|\mathbb{D}_2|)\frac{\mathbb{D}_2}{|\mathbb{D}|_2} \cdot \mathbb{D}_2\\
    &\ge f(|\mathbb{D}_1|)|\mathbb{D}_1| -f(|\mathbb{D}_1|)|\mathbb{D}_2|-f(|\mathbb{D}_2|)| \mathbb{D}_1|+f(|\mathbb{D}_2|)|\mathbb{D}_2|\\
    &=(f(|\mathbb{D}_1|)-f(|\mathbb{D}_2|))(|\mathbb{D}_1| -|\mathbb{D}_2|).
 \end{align*}
Since $f$ is nondecreasing, the inequality  \eqref{dec:5} follows. The same procedure, applied to $\mathbb{S}_n$, gives \eqref{dec:5M}.
\end{proof}

\end{document}